\documentclass[12pt]{amsart}
\usepackage{amsfonts, amssymb, amsmath, amsthm, euscript, color}
\usepackage{geometry}
\usepackage{graphicx}
\usepackage{amssymb}
\usepackage{epstopdf}
\usepackage{url}
\usepackage[all]{xy}

\theoremstyle{plain}
\newtheorem{thm}{Theorem}[section]
\newtheorem{lem}[thm]{Lemma}
\newtheorem{prop}[thm]{Proposition}
\newtheorem{cor}[thm]{Corollary}

\theoremstyle{definition}
\newtheorem*{Ack}{Acknowledgement}
\newtheorem{deff}[thm]{Definition}
\newtheorem{rem}[thm]{Remark}

\newtheorem{remark}[thm]{Remark}

\theoremstyle{remark}

\def\Z{\mathbb{Z}}

\def\k{\ensuremath{\bold{k}}}

\def\g{\mathfrak{g}}

\def\n{\mathfrak{n}}
\def\s{\mathfrak{s}}

\newcommand*{\e}{\ensuremath{\varepsilon}}
\newcommand*{\ee}{\ensuremath{\epsilon}}

\newcommand*{\complex}{\mathbb{C}}

\newcommand*{\field}{\ensuremath{\bold{k}}}

\newcommand*{\gr}{\ensuremath{\text{\upshape gr}}}
\newcommand*{\Ext}{\ensuremath{\text{\upshape Ext}}}
\newcommand*{\Ker}{\ensuremath{\text{\upshape Ker}}}
\newcommand*{\Img}{\ensuremath{\text{\upshape Im}}}

\newcommand*{\ad}{\ensuremath{\text{\upshape ad}}}

\newcommand*{\Id}{\ensuremath{\text{\upshape Id}}}

\newcommand*{\HL}{\ensuremath{\text{\upshape H}}}

\def\dim{\operatorname{dim}}

\def\Chara{\operatorname{char}}

\def\f{\frac}

\begin{document}
\thispagestyle{empty}

\title{Classification of connected Hopf algebras of dimension $p^3$ I}

\author{Van C. Nguyen}
\author{Linhong Wang}
\author{Xingting Wang*}

\address{Department of Mathematics\\
Texas A\&M University\\
College Station, TX 77843}
\email{vcnguyen@math.tamu.edu}

\address{Department of Mathematics\\
Southeastern Louisiana University\\
Hammond, LA 70402}
\email{lwang@selu.edu}

\address{Department of Mathematics\\
University of Washington\\
Seattle, WA 98195}
\email{xingting@uw.edu}

\thanks{* corresponding author, partially supported by U.~S.~National Science Foundation [DMS0855743]. Research of the second author supported by the Louisiana BoR [LEQSF(2012-15)-RD-A-20].}



\begin{abstract}
Let $p$ be a prime, and $\k$ be an algebraically closed field of characteristic $p$.
In this paper, we provide the classification of connected Hopf algebras of dimension $p^3$, except for the case when the primitive space of the Hopf algebra is a two-dimensional abelian restricted Lie algebra.
Each isomorphism class is presented by generators $x,\, y,\, z$ with relations and Hopf algebra structures.
Let $\mu$ be the multiplicative group of $(p^2+p-1)$-th roots of unity. When the primitive space is one-dimensional and $p$ is odd, there is an infinite family of isomorphism classes, which is naturally parameterized by $\textbf{A}_{\k}^1/\mu$.
\end{abstract}

\maketitle

\section{Introduction}
\thispagestyle{empty}

The classification of finite-dimensional Hopf algebras over $\complex$ has been done for certain dimensions; see \cite{andruskiewitsch2000finite, andruskiewitsch2014finite, BeattieSurvey2009, BeaGar2011, BeaGar2012, ChengNg2011}.
Let $p$ be a prime.
In particular, pointed Hopf algebras over $\complex$ of dimension $p^3$ are classified independently by Andruskiewitsch and Schneider \cite{AndruskiewitschSchneider1998p3}, Caenepeel and Dascalescu\cite{CaenepeelDascalescu}, Stefan and Van Oystaeyen \cite{cstefan1998hochschild}.
Finite-dimensional connected Hopf algebras only appear over fields of positive characteristic.
Throughout, we assume that the base field is algebraically closed and of characteristic $p$.
The third author of this paper, in \cite{wang2012connected}, classified connected Hopf algebras of dimension $p^2$ using the theories of restricted Lie algebras and Hochschild cohomology of coalgebras.
Pointed Hopf algebras of dimension $p^2$ are classified in \cite{WangWang}.
In particular, graded, cocommutative, connected Hopf algebras of dimension $p^2$ and $p^3$ are classified by Henderson \cite{Hen} using Singer's theory  \cite{WMS} of extensions of connected Hopf algebras.

In this paper, following the method in \cite{wang2012connected}, we classify connected Hopf algebras of dimension $p^3$, except for the case when the primitive space of the Hopf algebra is a two-dimensional abelian restricted Lie algebra.
Our result shows that there is an infinite family of isomorphism classes which is parameterized by $\textbf{A}_{\k}^1/\mu$, where $\mu$ is the multiplicative group of $(p^2+p-1)$-th roots of unity.
Moreover, one noncocommutative isomorphism class occurs when the primitive space is nonabelian and two-dimensional.

{\bf Preliminary.}
Throughout the paper, we work over a base field $\k$, algebraically closed of prime characteristic $p>0$. We use the standard notation $(H,\,m,\,u,\,\Delta,\,\e,\,S)$ as in \cite{MO93} to denote a Hopf algebra $H$.
The \emph{primitive space} of $H$ is the set $P(H)=\{x\in H\, |\; \Delta(x)=x\otimes 1 + 1\otimes x\}$.
In characteristic $p$, the primitive space $P(H)$ is a restricted Lie algebra, where the restricted map is given by the $p$th map in $H$.
For the notation regarding restricted Lie algebras, we follow \cite{Jac}.
The \emph{coradical} $H_0$ of $H$ is the sum of all simple subcoalgebras of $H$.
The Hopf algebra $H$ is \emph{connected} if $H_0$ is one-dimensional.
For each $n\geq 1$, set  \[H_n=\Delta^{-1}(H\otimes H_{n-1} + H_0\otimes H).\]
The chain of subcoalgebras $H_0\subseteq H_1 \subseteq \ldots \subseteq H_{n-1}\subseteq H_n \subseteq\ldots $ is the \emph{coradical filtration} of $H$; see \cite[\S 5.2]{MO93}.
Any finite-dimensional connected Hopf algebra must have dimension $p^n$ for some integer $n\ge 0$ (see, e.g., \cite[Proposition 2.2 (7)]{wang2012connected}).

Our settings and main results are described as follows and proofs are provided in Sections 3, 4 and 5.
In Section 2, we include related results that are used in our proofs. A further discussion on a special type in our classification results is provided in Section 6.

{\bf Classification results.}
Let $H$ be a connected Hopf algebra of dimension $p^3$,
and $K$ be the Hopf subalgebra of $H$ generated by its primitive space.
By \cite[Proposition 13.2.3]{Swe}, $K\cong u(P(H))$, the restricted universal enveloping algebra of $P(H)$.
Hence, $\dim K=p,\, p^2,\, p^3$.
Our main results are divided into three cases: $\dim K=p$ (Theorem \ref{KpdimclassesALL}), $\dim K=p^2$ and $K$ is noncommutative (Theorem \ref{Kp2NCclassesAll}), and $\dim K=p^3$ (Theorem \ref{Kdimp3classesAll}).
Hopf algebras in isomorphism classes of $H$ are always presented in the form of
\[\k \langle x,\, y,\, z\rangle / I,\] where $I$ is an ideal generated by relations.
And the comultiplication is given by
\begin{align*}
&\Delta(x)=x\otimes 1+ 1\otimes x,\\
&\Delta(y)=y\otimes 1+ 1\otimes y+ Y,\\
&\Delta(z)=z\otimes 1+ 1\otimes z+ Z,
\end{align*}
for some elements $Y$ and $Z$ in $(\k \langle x,\, y,\, z\rangle / I) \otimes (\k \langle x,\, y,\, z\rangle / I)$.
In Theorems \ref{KpdimclassesALL}, \ref{Kp2NCclassesAll}, and \ref{Kdimp3classesAll}, the defining relations in $I$ and the elements $Y$ and $Z$ are given explicitly.
In order to express $Y$ and $Z$, we need the convention \[\omega(t)=\sum_{i=1}^{p-1}\frac{(p-1)!}{i!(p-i)!} t^i\otimes t^{p-i}.\]

When $K$ is $p$-dimensional, by Proposition \ref{FCLH}, there is a series of normal Hopf subalgebras $K\subset F \subset H$, where $\dim F=p^2$.
Then $K$ is generated by $x$, and $F$ is generated by $x, y$. By Theorem \ref{D2}, $Y$ can be chosen as $\omega(x)$. Moreover, by applying Theorem \ref{Cohomologylemma}, we can find $Z$ by computing the basis of $\HL^2(\k, F)$. Other algebraic relations are found correspondingly.

\begin{thm}\label{KpdimclassesALL}
Suppose that $\dim K=p$. If $p=2$ then there are five isomorphism classes. Otherwise, there are four isomorphism classes and one infinite family.
\begin{itemize}
\item[(A1)] $\k[x, y, z]/(x^p-x,\ y^p-y,\ z^p-z)$
\end{itemize}
with $Y=\omega(x)$ and $Z=\omega(x)[y\otimes 1+1\otimes y+ \omega(x)]^{p-1}+\omega(y)$,
\begin{itemize}
\item[(A2)] $\k[x, y, z]/(x^p,\ y^p-x,\ z^p-y)$,
\item[(A3)] $\k[x, y, z]/(x^p,\ y^p,\ z^p)$,
\item[(A4)] $\k[x, y, z]/(x^p,\ y^p,\ z^p-x)$,
\item[(A5)]
\hspace{.2in}
\begin{itemize}
\item[ $p=2$,] $\k\langle x, y, z\rangle /(x^2,\ y^2,\ [x,y],\, [x,z],\ [y,\ z]- x,\, z^2+xy)$,
\item[ $p>2$,]
$A(\beta):=\k\langle x, y, z\rangle /(x^p,\ y^p,\ [x,y],\, [x,z],\ [y,\ z]- x,\, z^p+x^{p-1}y-\beta x)$,
\end{itemize}
\end{itemize}
for some $\beta\in \k$ with $Y= \omega(x)$ and $Z=\omega(x)(y\otimes 1+1\otimes y)^{p-1}+\omega(y)$.

\noindent Any two Hopf algebras $A(\beta)$ and $A(\beta')$ are isomorphic if and only if $\beta'=\gamma \beta$ for some $(p^2+p-1)$-th root of unity $\gamma$.
\end{thm}

\begin{rem}
The infinite family $A(\beta)$ is naturally parameterized by $\textbf{A}_{\k}^1/\mu$, where $\mu$ is the multiplicative group of $(p^2+p-1)$-th roots of unity.
\end{rem}

When $K$ is $p^2$-dimensional and noncommutative, the structure of $K$ is given by Theorem \ref{D2} (5).
Moreover, $H$ is generated by $K$ and some nonprimitive element $z \in H\setminus K$.
The element $Z$ can be found by applying Theorem \ref{HCT}.
Then the classification result follows by direct computation for each case in Lemma \ref{zDelta}.

\begin{thm}\label{Kp2NCclassesAll}
When $K$ is  $p^2$-dimensional and noncommutative and the element $Y=0$, then there are three isomorphism classes.
\begin{itemize}
\item[(B1)] $\k\langle x,\ y,\ z\rangle/([x, y]-y,\, [x, z],\, [y, z], x^p-x,\ y^p,\ z^p)$ with $Z=\omega(y)$,
\item[(B2)]  $\k\langle x,\ y,\ z\rangle/([x, y]-y,\, [x, z],\, [y, z]-yf(x), x^p-x,\ y^p,\ z^p-z)$ with $Z=\omega(x)$ and $f(x)=\sum_{i=1}^{p-1}(-1)^{i-1}(p-i)^{-1}x^i$,
\item[(B3)]  $\k\langle x,\ y,\ z\rangle/([x, y]-y,\, [x, z]-z,\, [y, z]-y^2, x^p-x,\ y^p,\ z^p)$ with $Z=-2x\otimes y$, for $p>2$.
\end{itemize}
\end{thm}

When $K$ is  $p^3$-dimensional (i.e., $H$ is primitively generated), it is sufficient to classify restricted Lie algebras of dimension three.

\begin{thm}\label{Kdimp3classesAll}
When $\dim K=p^3$ and the elements $Y=Z=0$, then there are fifteen isomorphism classes and one finite parametric family.
\begin{itemize}
\item[(C1)] $\field[x,y,z]/(x^p-x,y^p-y,z^p-z)$,
\item[(C2)] $\field[x,y,z]/(x^p-y,y^p-z,z^p)$,
\item[(C3)] $\field[x,y,z]/(x^p,y^p-z,z^p)$,
\item[(C4)] $\field[x,y,z]/(x^p,y^p,z^p)$,
\item[(C5)] $\field\langle x,y,z\rangle/([x,y]-z,[x,z],[y,z],x^p,y^p,z^p)$,
\item[(C6)] $\field\langle x,y,z\rangle/([x,y]-z,[x,z],[y,z],x^p-z,y^p,z^p)$, for $p>2$,
\item[(C7)] $\field[x,y,z]/(x^p,y^p,z^p-z)$,
\item[(C8)] $\field[x,y,z]/(x^p-y,y^p,z^p-z)$,
\item[(C9)] $\field[x,y,z]/(x^p,y^p-y,z^p-z)$,
\item[(C10)] $\field\langle x,y,z\rangle/([x,y]-z,[x,z],[y,z],x^p,y^p,z^p-z)$,
\item[(C11)] $\field\langle x,y,z\rangle/([x,y]-y,[x,z],[y,z],x^p-x,y^p,z^p)$,
\item[(C12)] $\field\langle x,y,z\rangle/([x,y]-y,[x,z],[y,z],x^p-x,y^p-z,z^p)$,
\item[(C13)] $\field\langle x,y,z\rangle/([x,y]-y,[x,z],[y,z],x^p-x,y^p,z^p-z)$,
\item[(C14)] $\field\langle x,y,z\rangle/([x,y]-y,[x,z],[y,z],x^p-x,y^p-z,z^p-z)$.
\item[(C15)] $\field\langle x,y,z\rangle/([x,y]-z,[x,z]-x,[y,z]+y,x^p,y^p,z^p-z)$, for $p> 2$,
\item[(C16)] The parametric family
\[C(\lambda, \delta):= \field\langle x,y,z\rangle/([x,y],[x,z]-\lambda x,[y,z]-\lambda^{-1} y,x^p,y^p,z^p-\delta z),\] for some $\lambda \in \k^{\times}$ such that $\delta:=\lambda^{p-1}= \pm 1$.
\end{itemize}
In type \emph{(C16)}, two Hopf algebras $C(\lambda_1, \delta_1)$ and $C(\lambda_2, \delta_2)$ are isomorphic if and only if $\delta_1=\delta_2$ and $\lambda_1=\lambda_2$ or $\lambda_1 \cdot \lambda_2=1$.
\end{thm}

\begin{rem}
In the above classification results, Hopf algebras of types (A1)-(A5), (B1), (B2), and (C1)-(C16) are all co-commutative. The Hopf algebra of type (B3) is a new example of non-cocommutative and non-commutative Hopf algebra over a field of positive characteristic.
Regarding $p^3$-dimensional graded co-commutative connected Hopf algebras, our classification overlaps some of the types found in Theorem 3.3 and Theorem 3.6 of \cite{Hen}.
The detailed correspondence is given as follows: (A3)-(h-3), (C2)-(f), (C3)-(g-1), (C4)-(h-1) in \cite[Theorem 3.3]{Hen} and (C5)-(j) in \cite[Theorem 3.6]{Hen}.
\end{rem}

The only remaining case is when $K$ is $p^2$-dimensional and commutative. This case will be treated in another paper.
Our approach will be based on \cite{XWangAbelian}, where the author classifies all the connected Hopf algebras having large abelian primitive space by constructing a group action on a cohomological type of group.
It is shown that
the isomorphism classes are in one-to-one correspondence with the group orbits in this cohomological type of group; see \cite[Theorem B]{XWangAbelian}.
For the remaining case, we expect more parametric families in the classification of connected $p^3$-dimensional Hopf algebras.

At last, we point out a fact that shows a difference between finite-dimensional connected Hopf algebras over characteristic $p$ with large primitive space and infinite-dimensional connected Hopf algebras over characteristic 0 with large primitive space.
According to \cite{BZ1,BZ2,WZZ2,Zh1} with respect to algebra structures, up to GK-dimension 4, affine connected Hopf algebras are all isomorphic to universal enveloping algebras.
Moreover, D.-G. Wang, J.J. Zhang, and G. Zhuang in \cite{WZZ2} studied the algebra structures of  connected Hopf algebras over an algebraically closed field of characteristic 0.
Under the assumption that when the GK-dimension equals the dimension of the primitive space plus one, it is proved that the algebra structure of Hopf algebra is always isomorphic to some universal enveloping algebra; see \cite[Theorem 0.5 or Theorem 2.7]{WZZ2}.
For finite-dimensional connected Hopf algebras in positive characteristic, we do not expect the similar algebraic property when we replace the universal enveloping algebra by its restricted version.
According to \cite[Theorems 7.1\&7.4]{wang2012connected}, any connected Hopf algebra of dimension less than $p^3$ is isomorphic to some restricted enveloping algebra with respect to algebra structures.
But when we move to dimension $p^3$, the (B2) type in Theorem \ref{Kp2NCclassesAll} gives us a counterexample. This is proved in Section 6.

\section{Background for finite-dimensional connected Hopf algebras}

In this section, we list some results from \cite{wang2012connected}. These are general results for finite-dimensional connected Hopf algebras over a field of positive characteristic, which will be used frequently in our proofs.

\begin{deff}\cite[Definition 2.2]{wang2012connected}\label{gr}
Let $H$ be a Hopf algebra with antipode $S$. If
\begin{itemize}
\item[(1)] $H=\bigoplus_{n=0}^{\infty}H(n)$ is a graded algebra,
\item[(2)] $H=\bigoplus_{n=0}^{\infty}H(n)$ is a graded coalgebra,
\item[(3)] $S(H(n))\subseteq H(n)$ for any $n\ge 0$,
\end{itemize}
then $H$ is called a \emph{graded Hopf algebra}.\ If in addition,
\begin{itemize}
\item[(4)] $H=\bigoplus_{n=0}^{\infty} H(n)$ is a coradically graded coalgebra,
\end{itemize}
then $H$ is called a \emph{coradically graded Hopf algebra}. Also, the \emph{associated graded Hopf algebra} of $H$ is defined by $\gr H=\bigoplus_{n\ge 0} H_n/H_{n-1}$ ($H_{-1}=0$) with respect to its coradical filtration.
\end{deff}

\begin{deff}\label{BDCHA} \cite[Definition 2.3]{wang2012connected}
Consider an inclusion of finite-dimensional connected Hopf algebras $K\subseteq H$.
\begin{itemize}
\item[(1)]  If $\dim K=p^m$ and $\dim H=p^n$, then the \emph{$p$-index} of $K$ in $H$ is defined to be $n-m$.
\item[(2)]  The \emph{first order} of the inclusion is defined to be the minimal integer $n$ such that $K_n\subsetneq H_n$. And we say it is infinity if $K=H$.
\item[(3)]  The inclusion is said to be \emph{level-one} if $H$ is generated by $H_n$ as an algebra, where $n$ is the first order of the inclusion.
\item[(4)]  The inclusion is said to be \emph{normal} if $K$ is a normal Hopf subalgebra of $H$.
\end{itemize}
\end{deff}

\begin{prop}\label{Contraddim} \cite[Lemma 4.1]{wang2012connected}
Suppose the inclusion of finite-dimensional connected Hopf algebras $K\subseteq H$ has first order $n$. Then the $p$-index of $K$ in $H$ is greater or equal to $\dim (H_n/K_n)$.
\end{prop}

\begin{prop}\label{FCLH} \cite[Corollary 5.3]{wang2012connected}
Let $H$ be a finite-dimensional connected Hopf algebra with $\dim P(H)=1$. Then $H$ has an increasing sequence of normal Hopf subalgebras:
\begin{equation*}
\field=N_0\subset N_1\subset N_2\subset \cdots \subset N_n=H,
\end{equation*}
where $N_1$ is generated by $P(H)$ and each $N_i$ has $p$-index one in $N_{i+1}$.
\end{prop}

\begin{thm}\label{NPLA}
Let $H$ be finite-dimensional cocommutative connected Hopf algebra. Denote by $K$ the Hopf subalgebra generated by $P(H)$. Then the following are equivalent:
\begin{itemize}
\item[(1)] $H$ is local.
\item[(2)] $K$ is local.
\item[(3)] All the primitive elements of $H$ are nilpotent.
\end{itemize}
\end{thm}
\begin{proof}
See \cite[Theorem 5.4]{wang2012connected} or \cite[Corollary 7.2]{XWangLocal}.
\end{proof}

We will also use \emph{the second term of the Hochschild cohomology} $\HL^2(\k, H)$ of $H$ with coefficients in the trivial $H$-bicomodule $\k$.
It can be computed as the homology of the following complex \cite[Lemma 1.1]{cstefan1998hochschild}:
\[
\xymatrix{
\k\ar[r]^-{0}&H\ar[r]^-{d^1}&H\otimes H\ar[r]^-{d^2}&H\otimes H\otimes H\ar[r]&\cdots,
}
\]
where the differentials $d^1$ and $d^2$ are defined as, for any $h,\, g\in H$,
\begin{align*}
d^1(h)&= 1\otimes h -\Delta(h) + h\otimes 1,\\
d^2(h\otimes g)&= 1\otimes h \otimes g-\Delta(h)\otimes g + h\otimes\Delta(g)- h\otimes g\otimes 1.
\end{align*}
Then
\[\HL^2(\k, H): =\Ker\, d^2/ \Img\, d^1.\]

\begin{prop}\label{DimExtCo}\cite[Lemma 6.1]{wang2012connected}
Let $H$ be a finite-dimensional Hopf algebra. Thus
\begin{eqnarray*}
\HL^2\left(\field,H\right)\cong \HL^2\left(H^*,\field\right)\cong \Ext^2_{H^*}\left(\field,\field\right).
\end{eqnarray*}
\end{prop}

\begin{prop}\label{Liealgebrainclusion}\cite[Proposition 6.2]{wang2012connected}
Let $\mathfrak g$ be a restricted Lie algebra with basis $\{x_1,x_2,\cdots,x_n\}$. Then the image of
\begin{eqnarray*}
\left\{\omega(x_i),\ x_j\otimes x_k\ |\ 1\le i\le n,1\le j<k\le n\right\}
\end{eqnarray*}
is a basis in $\HL^2\left(\field,u(\mathfrak g)\right)$.
\end{prop}

\begin{thm}\label{Cohomologylemma}\cite[Theorem 6.6]{wang2012connected}
Let $K\subseteq H$ be an inclusion of connected Hopf algebras with first order $n\ge 2$. Then the differential $d^1$ induces an injective restricted $\mathfrak g$-module map
\[
\xymatrix{
H_n/K_n\ar@{^(->}[r]&\HL^2(\field,K),
}\]
where $\mathfrak g=P(H)$.
\end{thm}

\begin{thm}\label{HCT}\cite[Theorem 6.7]{wang2012connected}
Let $\mathfrak g$ be a restricted Lie algebra with basis \[\{x_1, x_2, \cdots,x_n\}.\] Suppose $u(\mathfrak g)\subsetneq H$ is an inclusion of connected Hopf algebras. Then there exists some $x\in H\setminus u(\mathfrak g)$ such that
\begin{align*}
\Delta(x)=x\otimes 1+1\otimes x+\omega\left(\sum_i\alpha_ix_i\right)+\sum_{j<k}\alpha_{jk}x_j\otimes x_k
\end{align*}
with coefficients $\alpha_i,\alpha_{jk}\in \field$. Moreover,  the first order for the inclusion can only be $1$, $2$ or $p$.
\end{thm}

Stated in the next theorem is the classification of connected Hopf algebras over characteristic $p>0$ of dimension $p^2$.

\begin{thm}\label{D2} \cite[Theorem 7.4]{wang2012connected}
Let $H$ be a connected Hopf algebra of dimension $p^2$. When $\dim P(H)=2$, it is isomorphic to one of the following:
\begin{itemize}
\item[(1)] $\field\left[x,y\right]/\left(x^p,y^p\right)$,
\item[(2)] $\field\left[x,y\right]/\left(x^p-x,y^p\right)$,
\item[(3)] $\field\left[x,y\right]/\left(x^p-y,y^p\right)$,
\item[(4)] $\field\left[x,y\right]/\left(x^p-x,y^p-y\right)$,
\item[(5)] $\field\langle x,y\rangle/\left([x,y]-y,x^p-x,y^p\right)$,
\end{itemize}
where $x,y$ are primitive.
When $\dim P(H)=1$, it is isomorphic to one of the following:
\begin{itemize}
\item[(6)] $\field\left[x,y\right]/(x^p,y^p)$,
\item[(7)] $\field\left[x,y\right]/(x^p,y^p-x)$,
\item[(8)] $\field\left[x,y\right]/(x^p-x,y^p-y)$,
\end{itemize}
where $\Delta\left(x\right)=x\otimes 1+1\otimes x$ and $\Delta\left(y\right)=y\otimes 1+1\otimes y+\omega(x)$.
\end{thm}

\begin{thm}\label{centerP1} \cite[Theorem 7.7]{wang2012connected}
Let $H$ be a finite-dimensional connected Hopf algebra with $\dim P(H)=1$. Then the center of $H$ contains $P(H)$.
\end{thm}

Again, for the notation regarding restricted Lie algebras, we follow \cite{Jac}.
For instance, the adjoint mapping determined by an element $a$ in a restricted Lie algebra $\g$ is denoted by $\ad\, a$.
We have $x (\ad\, a)=[x, a]$ for any $x\in \g$.
Moreover, the adjoint notation is extended from $P(H)$ to $H$. That is, for any $x, y\in H$, we have $x (\ad\, y)=[x, y]=xy-yx$.
Also, the following proposition is used frequently in our proofs.

\begin{prop}\cite[P. 186-187]{Jac}\label{palgebra}
For any associative $\field$-algebra $A$, we have
\begin{equation*}
\left(x+y\right)^{p}=x^{p}+y^{p}+\sum_{i=1}^{p-1} s_i\left(x,y\right)
\end{equation*}
where $is_i(x,y)$ is the coefficient of $\lambda^{i-1}$ in $x(\ad\ (\lambda x+y))^{p-1}$ and
\begin{equation*}
\left[x^p,y\right]=(\ad\ x)^p(y)
\end{equation*}
for any $x,y\in A$.
\end{prop}

\section{$K$ is $p$-dimensional}
In this section, we suppose that $\dim K=p$, or equivalently $\dim P(H)=1$.
It is clear that $K$ is generated by a primitive element, which will be denoted by $x$.
Since $\dim P(H)=1$, it follows from  Theorem \ref{centerP1} that $x$ is in the center of $H$.
By Proposition \ref{FCLH}, $H$ is always cocommutative and contains a chain of normal Hopf subalgebras $K\subset F\subset H$, where $\dim F=p^2$.
In particular, $\dim P(F)=1$. Following the classification of connected Hopf algebras of dimension $p^2$ with one-dimensional primitive space provided in Theorem \ref{D2}, we can write $F$ as one of the following:
\begin{align*}\label{2a}\tag{3a}
 &\field[x,\ y]/(x^p-x,\ y^p-y),\\
 &\field[x,\ y]/(x^p,\ y^p-x),\\
 &\field[x,\ y]/(x^p,\ y^p),
\end{align*}
with
\begin{gather*}
\Delta(x)=x\otimes 1 + 1\otimes x,\quad \Delta(y)=y\otimes 1+1\otimes y+\omega(x),\\
\ee(x)=\ee(y)=0,\quad S(x)=-x, \quad S(y)=-y.
\end{gather*}

Next, we will obtain the comultiplication of the third generator of $H$.
By Proposition \ref{DimExtCo}, $\dim \HL^2(\field, F)=\dim \HL^2(F^*, \field)$. Direct computation shows that $F^*\cong \field[x]/(x^{p^2})$ as algebras. Then $\dim \HL^2(\field, F)=\dim \HL^2(F^*, \field)=1$.
Suppose $\HL^2(\field, F)$ is spanned by the image of a cocycle $u\in F\otimes F$.
Let $m$ be the first order of the inclusion $F\subset H$. Note that the $p$-index of $F$ in $H$ is one. Then $\dim H_m/F_m=1$ by Proposition \ref{Contraddim}. 
Furthermore, by Theorem \ref{Cohomologylemma}, $d^1$ induces a $\k$-vector space isomorphism from $H_m/F_m$ to $\HL^2(\field, F)$.
Therefore, there is some $z\in H\setminus F$ such that $d^1(z)=1\otimes z -\Delta(z)+z\otimes 1=-u$. That is, $\Delta(z)=z\otimes 1+1\otimes z+ u$. Note that the comultiplication for $F$ in the three cases of (\ref{2a}) are the same.
One can compute $u$ directly by using the Hopf algebra structures of $F$ in (\ref{2a}).

In fact, a similar computation has been done by Henderson in \cite{Hen} to classify low-dimensional connected graded Hopf algebras over fields of characteristic $p$.
In \cite[Theorem 3.3, type h-3]{Hen}, the connected graded Hopf algebra of dimension $p^3$ is generated by $x_1,\ y_1,\ z_1$ with
\[
\bar\psi(y_1)=\sum_{r,s> 0,\; r+s=p}\tfrac{1}{r!s!} x_1^r\otimes x_1^s
\]
and
\[
\bar\psi(z_1)=\sum_{r,s> 0,\; r+s=p}\tfrac{1}{r!s!} y_1^r\otimes y_1^s - \bar\psi(y_1)(\Delta(y_1))^{p-1},
\]
where $\bar\psi$ is the same as $-d^1$ in our notation (i.e., $\bar\psi(y_1)=\Delta(y_1) - y_1\otimes 1 - 1\otimes y_1$).
Note that the expression
\[\sum_{r,s> 0,\; r+s=p}\tfrac{1}{r!s!} x_1^r\otimes x_1^s\]
is the same as $-\omega(x_1)$, since $(p-1)!\equiv -1(\text{mod}\ p)$.
Thus, \[-d^1(z_1)=\bar\psi(z_1)=-\omega(y_1)+\omega(x_1)[\Delta(y_1)]^{p-1}.\]
The Hopf subalgebra $F_1$ generated by $x_1, y_1$ in \cite[Theorem 3.3 type h-3]{Hen} is isomorphic to $F$ as coalgebras via  $\phi: x_1\to x,\ y_1\to -y$.
Moreover, $[x_1,\ y_1]=[x,\ y]=0$.
Then, $\bar\psi(z_1)$, associated with the coalgebra structure, represents a basis in $\HL^2(\field, F_1)$.
Hence, back to our setting $F$ and $z\in H\setminus F$, we can choose $u$ as
\begin{align*}
(\phi \otimes \phi) (-d^1(z_1))
=&(\phi \otimes \phi) (-\omega(y_1)+\omega(x_1)[\Delta(y_1)]^{p-1})\\
=&-\omega(\phi(y_1))+\omega(\phi(x_1))[\Delta(\phi(y_1))]^{p-1}\\
=&\omega(y)+\omega(x)[\Delta(y)]^{p-1}.
\end{align*}
Note that $F$ and $z$ generate a Hopf subalgebra of $H$. Since $z\in H\setminus F$ and $F$ has $p$-index one in $H$, then $H$ is generated by $x,\ y,\ z$. Moreover, we can assume
\[
\Delta(z)=z\otimes 1 + 1\otimes z +\omega(x)[y\otimes 1+1\otimes y+\omega(x)]^{p-1}+\omega(y)
\]
for all the three cases in (\ref{2a}).

\begin{lem}\label{Kpdim_primitive_elt}
Let $x,\ y,\ F$ and $z$ be as above and further assume that $x^p=0$. Then

\noindent \emph{i)} $u= \omega(x)(y\otimes 1 + 1\otimes y)^{p-1} +\omega(y)$,

\noindent \emph{ii)} $[y,z]=\gamma x$ for some $\gamma \in \k$, and

$
\Delta(z^p+\gamma^{p-1}x^{p-1}y)=(z^p+\gamma^{p-1}x^{p-1}y)\otimes 1 + 1\otimes (z^p+\gamma^{p-1}x^{p-1}y) + \omega(y^p)
$.
\end{lem}
\begin{proof}
i) First, we rewrite $u= \omega(x)[y\otimes 1+1\otimes y+\omega(x)]^{p-1}+\omega(y)$ as
\[
\omega(x)(y\otimes 1 + 1\otimes y)^{p-1} + \sum_{i=1}^{p-1}\tbinom{p-1}{i} \omega(x)^{i+1}(y\otimes 1 + 1\otimes y)^{p-1-i}+\omega(y).
\]
Since $x^p=0$, it is easy to see that $\omega(x)^n=0$ for $n\ge 2$, and so \[
u= \omega(x)(y\otimes 1 + 1\otimes y)^{p-1} +\omega(y).
\]

ii) Note that $x$ is in the center of $H$. Then,  we have
\[
\Delta([y,\ z])=[\Delta(y),\Delta(z)]= [y,z]\otimes 1+1\otimes [y,z].
\]
Hence, $[y,z]$ is primitive in $H$. So we can write $[y,z]=\gamma x$ for some $\gamma \in \field$. For positive integers  $m$ and $n$, it follows by direct computation that
\begin{align*}\tag{3b}\label{2b}
(y^m)(\ad\ z)^n=
\begin{cases}
\frac{m!}{(m-n)!}\, \gamma^n x^n y^{m-n} & m\geq n,\\
0 & m< n.
\end{cases}
\end{align*}
By Proposition \ref{palgebra},
\begin{align*}
\Delta(z^p)&=\Delta(z)^p= (u + z\otimes 1+ 1\otimes z)^p\\
&=u^p+ (z\otimes 1+ 1\otimes z)^p + \sum_{i=1}^{p-1}s_i,
\end{align*}
where $i s_i$ is the coefficient of $\lambda^{i-1}$ in $u (\ad\ (\lambda u + z\otimes 1+ 1\otimes z))^{p-1}$.
Since $[x, z]=0$ and $[y,\ z]=\gamma x$,
we have $u (\ad\ (\lambda u + z\otimes 1+ 1\otimes z))^i\in (F\otimes F)[\lambda]$ for any integer $i\ge 0$.
Moreover, $F$ is commutative. Hence, for any $f\in F\otimes F$, we have
\[
f(\ad\ (\lambda u + z\otimes 1+ 1\otimes z)) = f(\ad\ z\otimes 1+ 1\otimes \ad\ z).
\]
Thus,
\[u (\ad\ (\lambda u + z\otimes 1+ 1\otimes z))^{p-1}=u (\ad\ z\otimes 1+ 1\otimes \ad\ z)^{p-1}.\]
By i),  $u^p=[\omega(x)(y\otimes 1 + 1\otimes y)^{p-1} +\omega(y)]^p =\omega(y)^p=\omega(y^p)$. Therefore,
\[
\Delta(z^p)= z^p\otimes 1+ 1\otimes z^p+\omega(y^p)+ u (\ad\ z\otimes 1+ 1\otimes \ad\ z)^{p-1}.
\]
The last term of $\Delta(z^p)$ is the sum of $[\omega(x)(y\otimes 1 + 1\otimes y)^{p-1}](\ad\ z\otimes 1+ 1\otimes \ad\ z)^{p-1}$ and $\omega(y)(\ad\ z\otimes 1+ 1\otimes \ad\ z)^{p-1}$,
which we will refer to as $I_1$ and $I_2$, respectively.

First of all, $I_1=\omega(x)[(y\otimes 1 + 1\otimes y)^{p-1}(\ad\ z\otimes 1+ 1\otimes \ad\ z)^{p-1}]$, since $x$ is in the center of $H$.
After binomial expansion,
\[
I_1=\omega(x)\sum_{i, j=0}^{p-1}\, \tbinom{p-1}{i}\tbinom{p-1}{j}\, y^i(\ad\  z)^j \otimes y^{p-1-i} (\ad\ z)^{p-1-j}.
\]
By (\ref{2b}), the terms for $i\neq j$ in the above summation all vanish.
Hence,
\begin{align*}
I_1&=\omega(x)\sum_{i=0}^{p-1}\, \tbinom{p-1}{i}\tbinom{p-1}{i}\, y^i(\ad\ z)^i \otimes y^{p-1-i} (\ad\ z)^{p-1-i}\\
&=\omega(x)\sum_{i=0}^{p-1}\, \frac{(p-1)!^2}{i! (p-1-i)!}\,\gamma^{p-1}\, x^i\otimes x^{p-1-i}.
\end{align*}
Note that $(p-1)!=-1\ (\text{mod}\ p)$. Then
\[
I_1=-\gamma^{p-1}\, \omega(x)\, (x\otimes 1+1\otimes x)^{p-1}=-\gamma^{p-1}\, \omega(x)\, \Delta(x)^{p-1}.
\]
For $I_2$, we follow a similar argument.
By direct expansion,
\[
I_2=\sum_{i=1}^{p-1}\sum_{j=0}^{p-1}\, t(i,j),
\] where
\[t(i,j)=\frac{(p-1)!}{i! (p-i)!} \tbinom{p-1}{j} \, y^i(\ad\ z)^j \otimes y^{p-i} (\ad\ z)^{p-1-j}.\]
Again by (\ref{2b}), the terms for $i\ne j$ or $i\ne j+1$ all vanish. Then, we have
{\small
\begin{align*}
I_2=&t(1,0)+t(p-1,\, p-1)+\sum_{j=1}^{p-2}[t(j,j)+t(j+1, j)]\\
=&-\gamma^{p-1}y\otimes x^{p-1}
-\gamma^{p-1}x^{p-1}\otimes y
-\gamma^{p-1}\sum_{j=1}^{p-2}
\Big[\tbinom{p-1}{j}  x^j\otimes x^{p-1-j}y +
\tbinom{p-1}{j}  x^jy\otimes x^{p-1-j}\Big]\\
=&-\gamma^{p-1} [(x\otimes 1 + 1\otimes x)^{p-1}(y\otimes 1+ 1\otimes y)-x^{p-1}y\otimes 1 - 1\otimes x^{p-1}y].
\end{align*}
}
Thus,
\begin{align*}
\Delta(z^p)&= z^p\otimes 1+ 1\otimes z^p+\omega(y^p)+I_1 +I_2\\
&=z^p\otimes 1+ 1\otimes z^p+\omega(y^p) -\Delta(\gamma^{p-1}x^{p-1}y)+(\gamma^{p-1}x^{p-1}y\otimes 1 + 1\otimes \gamma^{p-1}x^{p-1}y).
\end{align*}
The second assertion of ii) follows immediately.
\end{proof}
\begin{thm}\label{Kpdimclasses}
Suppose that $x$ and $y$ are described in \emph{(\ref{2a})}, that is,
\begin{gather*}
\Delta(x)=x\otimes 1 + 1\otimes x,\quad \Delta(y)=y\otimes 1+1\otimes y+\omega(x),\\
\ee(x)=\ee(y)=0,\quad S(x)=-x, \quad S(y)=-y.
\end{gather*}
Then $H$ is isomorphic to
\begin{itemize}
\item[(A1)] $\field[x,\ y,\ z]/(x^p-x,\ y^p-y,\ z^p-z)$ with
\begin{gather*}
\Delta(z)=z\otimes 1+1\otimes z+\omega(x)[y\otimes 1+1\otimes y+\omega(x)]^{p-1}+\omega(y),\\
\ee(z)= 0 , \quad S(z)= -z,
\end{gather*}
\end{itemize}
or one of the following
\begin{itemize}
\item[(A2)] $H(\alpha):= \field[x, y, z]/(x^p,\ y^p-x,\ z^p-y-\alpha x)$,
\item[(A3)] $\field[x, y, z]/(x^p,\ y^p,\ z^p)$,
\item[(A4)] $\field[x, y, z]/(x^p,\ y^p,\ z^p-x)$,
\item[(A5)] $A(\beta):=\field \langle x, y, z\rangle /(x^p,\ y^p,\ [x,y],\,  [x,z],\ [y,z]-x,\, z^p+x^{p-1}y-\beta x)$,
\end{itemize}
where $\alpha, \beta\in \field$ and
\begin{gather*}
\Delta(z)=z\otimes 1+1\otimes z+\omega(x)(y\otimes 1+1\otimes y)^{p-1}+\omega(y),\\
\ee(z)=0 , \quad S(z)= -z.
\end{gather*}
\end{thm}
\begin{proof}
Retain the information on $x,\ y,\ F$, and $z$ obtained prior to Lemma \ref{Kpdim_primitive_elt}. Then $H$ is generated by $F$ and $z$. There are three cases according to (\ref{2a}).

\textbf{Case 1.} Suppose that $F=\field[x,\ y]/(x^p-x,\ y^p-y)$.
Note that $P(H)$ is a \emph{torus} in the sense of \cite{Masu}.
Then by \cite[Theorem 0.1]{Masu}, $H$ is commutative and semisimple. In particular, $[y,\ z]=0$.
Consider $u^p=\{\omega(x)[y\otimes 1+1\otimes y+\omega(x)]^{p-1}+\omega(y)\}^p$. It follows from $[x, y]=0$, $x^p=x$, $y^p=y$, and $\Chara \field=p$ that $u^p=u$. Hence
\[
\Delta(z^p-z)=\Delta(z)^p-\Delta(z)=(z^p-z)\otimes 1+1\otimes (z^p-z).
\]
That is, $z^p-z$ is primitive. There is some $\lambda\in \field$ such that $z^p-z=\lambda x$.
Now consider the shifting $z'=z+\alpha x$ for some $\alpha \in \k$.
It is clear that $\Delta(z')=z'\otimes 1 + 1\otimes z' + u$. If the scalar $\alpha$ satisfies $\lambda+\alpha^p-\alpha=0$, then we have $z'^p-z'=z^p-z+(\alpha^p-\alpha)x=0$.
In other words, we can assume, up to isomorphism, that $z^p=z$. This gives (A1).


\textbf{Case 2.} Suppose $F=\field[x,\ y]/(x^p,\ y^p-x)$.
Note that $x^p=y^p-x=0$.  By Lemma \ref{Kpdim_primitive_elt} (ii), we can write $[y,\ z]=\gamma x$ for some $\gamma \in \k$ and \[\Delta(z^p+\gamma^{p-1}x^{p-1}y)=(z^p+\gamma^{p-1}x^{p-1}y)\otimes 1 + 1\otimes (z^p+\gamma^{p-1}x^{p-1}y) + \omega(x).\]
Moreover, $\Delta(y)=y\otimes 1 + 1\otimes y +\omega(x)$.
Hence $z^p+\gamma^{p-1}x^{p-1}y-y$ is primitive. Thus, $z^p+\gamma^{p-1}x^{p-1}y=y+\alpha x$ for some $\alpha \in \field$.
Note that $[z^p+\gamma^{p-1}x^{p-1}y,\ z]=\gamma^{p-1}x^{p-1}[y,\ z]=\gamma^{p}x^{p}=0$ and that $[y+\alpha x,\ z]=[y,\ z]=\gamma x$. Thus, $\gamma=0$, $[y,\ z]=0$, and so $z^p=y+\alpha x$. This gives (A2).


\textbf{Case 3.} Suppose $F=\field[x,\ y]/(x^p,\ y^p)$. By Lemma \ref{Kpdim_primitive_elt} ii), we can write $[y,\ z]=\gamma x$ for some $\gamma \in \k$ and $z^p+\gamma^{p-1}x^{p-1}y$ is primitive, since $\omega(y^p)=0$. So we can write $z^p+\gamma^{p-1}x^{p-1}y=\alpha x$ for some $\alpha \in \field$.

When $\gamma=0$ and $\alpha=0$, it follows that $[y,\ z]=0$ and $z^p=0$. This gives (A3).

When $\gamma=0$ and $\alpha\ne 0$, consider the rescaling $x'=ax,\ y'=a^py,\ z'=a^{p^2}z$ of the variables $x,\ y,\ z$, where $a\in \field^{\times}$.
One can check that the rescaling preserves the comultiplication on the generators, that is, $\Delta(x')=x'\otimes 1 + 1\otimes x'$, $\Delta(y')=y'\otimes 1+1\otimes y'+\omega(x')$, and $\Delta(z')=z'\otimes 1+1\otimes z'+\omega(x')[y'\otimes 1+1\otimes y']^{p-1}+\omega(y')$.
It is clear that $x'^p=0$, $y'^p=0$, and $[y',\ z']=0$. Let $a=\alpha^{-p^3+1}$. Then $z'^p=a^{p^3}z^p=a^{p^3}\alpha x=a^{p^3-1}\alpha x'=x'$.
Therefore, we can assume, up to isomorphism, that $z^p=x$. This gives (A4).

When $\gamma\ne 0$,  consider again the rescaling $x'=ax,\ y'=a^py,\ z'=a^{p^2}z$ of the variables $x,\ y,\ z$, for some $a\in \field^{\times}$.
It is clear that $x'^p=0$, $y'^p=0$. Let $a=\gamma^{-1 / (p^2+p-1)}$.
Then $[y',\ z']=a^{p^2+p}[y,\ z]=a^{p+p^2}\gamma x =a^{p^2+p-1}\gamma x'=x'$.
Hence, up to isomorphism, we can assume that $\gamma=1$. This gives (A5).
\end{proof}

\begin{rem} The Hopf algebra in Theorem \ref{Kpdimclasses} (A1), namely, $H=\field[x,\ y,\ z]/(x^p-x,\ y^p-y,\ z^p-z)$, is semisimple and isomorphic to $(\field C_{p^3})^*$, where $C_{p^3}$ is the cyclic group of order $p^3$.
\end{rem}

Notice that the Hopf algebras in Theorem \ref{Kpdimclasses} (A2) and (A5) are subject to parameters $\alpha$ and $\beta$, respectively. Next, we discuss the isomorphism classes with respect to the choices of the parameters.

\begin{prop}\label{KpParameterClasses}
Let $A(\beta)$ and $A(\beta')$ be Hopf algebras described in Theorem \ref{Kpdimclasses} \emph{(A5)}, $H(\alpha)$ and $H(\alpha')$ be Hopf algebras described in Theorem \ref{Kpdimclasses} \emph{(A2)}. Then

\emph{(i)} When $p=2$,  $A(\beta)\cong A(\beta')$ for any $\beta$ and $\beta'$ in $\k$. When $p>2$, $A(\beta)\cong A(\beta')$ if and only if there exists a $(p^2+p-1)$-th root of unity $\gamma$ such that $\beta'=\gamma \beta$.

\emph{(ii)} $H(\alpha)\cong H(\alpha')$ for any $\alpha$ and $\alpha'$ in $\k$.

\end{prop}
\begin{proof}
(i) Denote the generators of $A(\beta')$ and $A(\beta)$ by $x', y', z'$ and $x, y, z$, respectively.
Suppose $A(\beta')$ and $A(\beta)$ are isomorphic as Hopf algebras via $\phi:\ A(\beta')\to A(\beta)$.
We first study the properties of such $\phi$ as a Hopf algebra isomorphism.

To start, $\phi$ induces an isomorphism $\phi_{\gr}: \gr A(\beta')\to \gr A(\beta)$, where $\gr A(\beta')$ and $\gr A(\beta)$ are associated graded Hopf algebras with respect to the coradical filtration; see Definition \ref{gr}.
It is clear that
\[
\gr A(\beta)=\field \langle \bar{x},\bar{y},\bar{z}\rangle / (\bar{x}^p,\bar{y}^p,\bar{z}^p)
\]
with comultiplication given by
\begin{align*}
\Delta(\bar{x})&=\bar{x}\otimes 1+1\otimes \bar{x},\\
\Delta(\bar{y})&=\bar{y}\otimes 1+1\otimes \bar{y}+ \omega(\bar{x}),\\
\Delta(\bar{z})&=\bar{z}\otimes 1+1\otimes \bar{z}+\omega(\bar{x})(\bar{y}\otimes 1+1\otimes \bar{y})^{p-1}+\omega(\bar{y}).
\end{align*}
Hence, there exists some scalar $\gamma \in \field^{\times}$ such that
\[
\phi_{\gr}(\bar{x'})=\gamma \bar{x},\quad
\phi_{\gr}(\bar{y'})=\gamma^p \bar{y},\quad
\phi_{\gr}(\bar{z'})=\gamma^{p^2} \bar{z}.
\]
Now, back to the map $\phi: A(\beta')\to A(\beta)$.
Since the first order of the inclusion $K\subset F$ is $p$, we have $\phi(y')=\gamma^py+A$ for some $A\in F_{p-1}=K_{p-1}\subset K$. Similarly, we have $\phi(z')=\gamma^{p^2}z+B$ for some $B\in F$. That is,
\[
\phi(x')=\gamma x,\quad
\phi(y')=\gamma^py+A,\quad
\phi(z')=\gamma^{p^2}z+B,
\]
where $A\in K$ and $B\in F$.

Next, since $\phi$ is a coalgebra map, $(\phi \otimes \phi)\Delta(y')=\Delta(\phi(y'))$. Then, it can be shown that $\Delta(A)=A\otimes 1+1\otimes A$.
Hence, we can write $A=ax$ for some $a\in \k$.
Similarly, it follows from $(\phi \otimes \phi)\Delta(z')=\Delta(\phi(z'))$ that
\begin{align*}\tag{3c}\label{2c}
&\Delta(B)-B\otimes 1-1\otimes B\\
=&\gamma^p\omega(x)[\gamma^p(y\otimes 1+1\otimes y)+a(x\otimes 1+1\otimes x)]^{p-1} +\omega(\gamma^py+ax)\\
&-\gamma^{p^2}\omega(x)(y\otimes 1+1\otimes y)^{p-1}-\gamma^{p^2}\omega(y).
\end{align*}
Direct computation shows that the element
\[
\sum_{i=1}^{p-1}\frac{(p-1)!}{i!(p-i)!}(\gamma^py)^i(ax)^{p-i}+ a^py
\]
satisfies the above equation (\ref{2c}) in $B$. (The computation is tedious and omitted here.)
If $B_1$ and $B_2$ are both solutions to the equation (\ref{2c}), then $d^1(B_1-B_2)=0$, and so $B_1$ and $B_2$ differ by a primitive element.
Thus, we can write
\[B=M+ a^py +bx, \text{ where } b\in \k  \text{ and } M:=\sum_{i=1}^{p-1}\frac{(p-1)!}{i!(p-i)!}(\gamma^py)^i(ax)^{p-i}.
\]
It follows immediately that $M\in F$ and $M^p=0$.
In summary, the isomorphism $\phi$ can be written as:
\begin{align*}\tag{3d}\label{2d}
\left\{
\begin{array}{l}
\phi(x')=\gamma x,\\
\phi(y')=\gamma^py+ax,\\
\phi(z')=\gamma^{p^2}z+M+ a^py +bx,
\end{array}
\right.
\end{align*}
for some $a,b\in \k$ and $\gamma\in \k^\times$.

Finally, we consider the fact that $\phi$ is also an algebra map. In particular,
\begin{align*}
\phi([y',z']-x')&=0,\\
\phi\Big((z')^p+(x')^{p-1}y'-\beta' x'\Big)&=0.
\end{align*}
We claim that
\[\phi(z')^p=
\begin{cases}
\gamma^{p^3}z^p & \text{for } p>2,\\
\gamma^8z^2+ \gamma^4a^2x& \text{for } p=2.
\end{cases}
\]
Since $[x,y]=[x,z]=x^p=y^p=0$, we have
\begin{align*}
\phi(z')^p
&=(\gamma^{p^2}z+ M +a^py + bx)^p\\
&=(\gamma^{p^2}z+ M +a^py)^p\\
&=(\gamma^{p^2}z+ M)^p+ S,
\end{align*}
where $S=\sum_{i=1}^{p-1}s_i$ with $is_i$ being the coefficient of $\lambda^{i-1}$ in
\[(a^py) (\ad\ (\lambda a^py + \gamma^{p^2}z+ M))^{p-1}.\]
Note that
\[[a^py, \lambda a^p y + \gamma^{p^2} z + M]= \gamma^{p^2}a^p x.\]
Then, since $x$ is in the center of $A(\beta)$, we have $S=0$ for $p>2$.
When $p=2$, we have $S=\gamma^4 a^2 x$.
By (\ref{2b}) and the fact that $x^p=0$, one can check that
\[(y^ix^{p-i})(\ad\ z)^{p-1}=x^{p-i}[y^i (\ad\ z)^{p-1}]=0.\]
Hence $M (\ad\ (\gamma^{p^2} z))^{p-1}=0$, and so
\[(\gamma^{p^2}z+ M)^p=\gamma^{p^3}z^p.\]
This proves the claim. On the other hand, the condition $\phi([y',z']-x')=0$ implies that
\[
\left[\gamma^py+ax,\quad \gamma^{p^2}z+ M+ a^py +bx\right]=\gamma x.
\]
In view of the algebraic relations in $A(\beta)$, this yields \[\gamma^{p^2+p-1}=1.\]
Moreover, $\phi\Big((z')^p+(x')^{p-1}y'-\beta' x'\Big)=0$
implies that $\phi(z')^p+ \gamma^{2p-1}x^{p-1}y-\beta'\gamma x=0$. Combining with the relation $z^p+x^{p-1}y-\beta x=0$, we have
\begin{align*}
\phi(z')^p-\gamma^{2p-1}z^p+ (\beta\gamma^{2p-1}-\beta'\gamma) x=0.
\end{align*}


In summary, we see that any isomorphism from $A(\beta')\rightarrow A(\beta)$ must be in the form of (\ref{2d}) for some $a,b\in \k$, $\gamma\in \k^\times$, \[M=\sum_{i=1}^{p-1}\frac{(p-1)!}{i!(p-i)!}(\gamma^py)^i(ax)^{p-i},\]
and satisfy the following condition:
\begin{align*}\tag{3e}\label{2e}
\begin{cases}
\gamma^{p^2+p-1}=1,\\
\phi(z')^p-\gamma^{2p-1}z^p+ (\beta\gamma^{2p-1}-\beta'\gamma) x=0,
\end{cases}
\end{align*}
where
\[\phi(z')^p=
\begin{cases}
\gamma^{p^3}z^p & \text{for } p>2,\\
\gamma^8z^2+ \gamma^4a^2x& \text{for } p=2.
\end{cases}
\]

We can check that any $\k$-linear map from $A(\beta')\rightarrow A(\beta)$ as described above for some $a,b\in \k$, $\gamma\in \k^\times$ preserves the algebra and coalgebra structure and so is a bialgebra map.
By \cite[Proposition 4.2.5]{dascalescu2000hopf}, any bialgebra map is always a Hopf algebra map.
Moreover, any such a map induces an isomorphism between the associated graded algebra of $A(\beta')$ and $A(\beta)$ with respect to the coradical filtration, and then becomes a Hopf algebra isomorphism.
Therefore, any map given in (\ref{2d}) is an isomorphism from $A(\beta')$ to $A(\beta)$ if and only if (\ref{2e}) also holds.

When $p>2$, it follows from (\ref{2e}) that
\[
\gamma^{p^3-2p+1}=1\quad \text{and} \quad
\beta'=\beta \gamma^{2p-2}.
\]
Note that $p^2+p-1$ divides $p^3-2p+1$ and that $(p^2+p-1,\, 2p-2)=1$ for $p>2$. Therefore, $A(\beta')\cong A(\beta)$ if and only if there is a map $\phi:\, A(\beta') \to A(\beta)$ in the form of (\ref{2d}) for some $a, b\in \k$ and a $(p^2+p-1)$-th root of unity $\gamma$ such that $\beta'=\gamma \beta$.

When $p=2$, it follows from (\ref{2e}) that
\[\gamma^8z^2+ \gamma^4a^2x - \gamma^3z^2+\beta\gamma^3-\beta'\gamma=0.\]
Hence, $A(\beta')\cong A(\beta)$ if and only if there is a map $\phi:\, A(\beta') \to A(\beta)$ in the form of (\ref{2d}) with $\gamma^5=1$ and $a^2=\beta'\gamma^{-3}-\beta\gamma^{-1}$. In other words, when $p=2$, $A(\beta') \cong A(\beta)$ for any $\beta, \beta'\in \k$.


(ii) Denote the generators of $H(\alpha')$ and $H(\alpha)$ by $x', y', z'$ and $x, y, z$, respectively.
Similar computation used in (i) shows that the following map $\psi$ is an isomorphism from $H(\alpha')$ to $H(\alpha)$
\begin{align*}
\left\{
\begin{array}{l}
\psi(x')= x,  \\
\psi(y')=y+ax,\\
\displaystyle{\psi(z')=z+\sum_{i=1}^{p-1} \frac{(p-1)!}{i!(p-i)!}y^i (ax)^{p-i}+a^py},
\end{array}
\right.
\end{align*}
where $a\in \k$ such that $a^2-a=\alpha'-\alpha$. Therefore, $H(\alpha')\cong H(\alpha)$ for any $\alpha, \alpha'\in \k$.
\end{proof}


\noindent \textbf{Proof of Theorem \ref{KpdimclassesALL}.}
In Theorem \ref{Kpdimclasses}, the type (A5) is the only one with noncommutative algebra structure.
It is clear that (A1) to (A4) are not isomorphic to each other as commutative algebras.
Moreover, in Proposition \ref{KpParameterClasses}, let $\beta=0$ when $p=2$ in (A5) and $\alpha=0$ in (A2).
Then the theorem follows by combining Theorem \ref{Kpdimclasses} and Proposition \ref{KpParameterClasses}.
\section{$K$ is $p^2$-dimensional and noncommutative}
In this section, suppose $\dim K=p^2$ and $K$ is noncommutative. By Theorem \ref{D2} (5), we can write
\[K=\field\langle x,y\rangle\big/([x,y]-y,x^p-x,y^p),\]
where $x$ and $y$ are primitive elements with $\ee(x)=\ee(y)=0$, $S(x)=-x$, and $S(y)=-y$. The following lemma contains some computational results needed later. We also use the convention
\[ f(x)= \sum_{i=1}^{p-1}\frac{(-1)^{i-1}}{(p-i)} x^i.\]

\begin{lem}\label{Kp2calc} Let $x,\ y$ and $K$ as above. Then we have

\noindent \emph{i)} $[x\otimes 1 + 1\otimes x, \omega(y)]=0$,

\noindent \emph{ii)} $[y\otimes 1+1\otimes y,\omega(x)]= \Delta(y f(x))- y f(x)\otimes 1 - 1 \otimes y f(x)$.

\noindent \emph{iii)} If $e\in H$ is primitive such that $[x,\ e]=[y,\ e]=0$, then $e=0$.
\end{lem}
\begin{proof}
i) Note that  $xy=yx+y$. Then, by induction, $[x,y^i]=iy^i$ for any positive integer $i$. Hence
\[
[x\otimes 1 + 1\otimes x, y^i\otimes y^{p-i}]=[x, y^i]\otimes y^{p-i} + y^i\otimes [x, y^{p-i}]=i y^i\otimes y^{p-i} + y^i\otimes (p-i)y^{p-i}=0.
\]
Therefore,
\[
[x\otimes 1 + 1\otimes x, \omega(y)]=  \sum_{i=1}^{p-1}\frac{(p-1)!}{i!(p-i)!} [x\otimes 1 + 1\otimes x, y^i\otimes y^{p-i}]=0.
\]


ii) Again by $xy=yx+y$, one can show inductively that
\[[y, x^i] = - [x^i,y]= - \sum_{j=0}^{i-1}\tbinom{i}{j}y x^j\]
for all $i\ge 1$. Note that the index $i$ and $p-i$ are symmetric in $\omega(x)$. Hence,
\begin{eqnarray*}
[y\otimes 1+1\otimes y,\omega(x)]&=&\sum_{i=1}^{p-1}\frac{(p-1)!}{i!(p-i)!} \left([y,x^{p-i}]\otimes x^{i}+x^{i}\otimes [y,x^{p-i}]\right)\\
&=& -\sum_{i=1}^{p-1}\frac{(p-1)!}{i!(p-i)!} \sum_{j=0}^{p-i-1}\tbinom{p-i}{j} \Big(yx^j\otimes x^{i}+x^{i}\otimes yx^j\Big).
\end{eqnarray*}
Moreover, since $\Chara \field =p$, we have
\[\f{(p-1)!}{(p-1-i-j)!}=(-1)^{i+j} (i+j)!.\]
It follows by direct computation that for any indices $i, j$ the coefficient $\f{(p-1)!}{i!(p-i)!} \tbinom{p-i}{j}$ can be rewritten as $\f{(-1)^{i+j}}{p-i-j} \tbinom{i+j}{i}$.
Hence
\[
[y\otimes 1+1\otimes y,\omega(x)]=\sum_{i=1}^{p-1} \sum_{j=0}^{p-i-1} \frac{(-1)^{i+j+1}}{p-i-j} \tbinom{i+j}{i} \Big(yx^j\otimes x^{i}+x^{i}\otimes yx^j\Big).
\]
To find $\Delta(y f(x))$, we first define
\[
A_l= (y\otimes 1+ 1\otimes y)\Big[(x\otimes 1+ 1\otimes x)^{l}- x^{l}\otimes 1 - 1\otimes x^{l}\Big],
\]
for positive integer $l$, which can be rewritten as
\begin{align*}
(y\otimes 1+ 1\otimes y) \sum_{m=1}^{l-1}\tbinom{l}{m} x^{l-m}\otimes  x^{m}
&= \sum_{m=1}^{l-1}\tbinom{l}{m} \Big(yx^{l-m}\otimes  x^{m} + x^{l-m}\otimes y x^{m}\Big)\\
&= \sum_{m=1}^{l-1}\tbinom{l}{m} \Big(yx^{l-m}\otimes  x^{m} + x^{m}\otimes y x^{l-m}\Big).
\end{align*}
Then we obtain
\begin{align*}
\Delta(y f(x))&= \Delta(y)f(\Delta(x)) = (y\otimes 1+ 1\otimes y) \Big[\sum_{l=1}^{p-1}\frac{(-1)^{l-1}}{p-l} (x\otimes 1+ 1\otimes x)^{l}\Big]\\
&= \sum_{l=1}^{p-1}\frac{(-1)^{l-1}}{p-l}\Big[ (y\otimes 1+ 1\otimes y) \Big(x^l\otimes 1 +1\otimes x^l\Big) + A_l \Big]\\
&= \sum_{l=1}^{p-1}\frac{(-1)^{l-1}}{p-l} \Big[ yx^l\otimes 1 +1\otimes yx^l + \sum_{m=1}^{l} \tbinom{l}{m}\Big(yx^{l-m}\otimes  x^{m} + x^{m}\otimes y x^{l-m}\Big) \Big]\\
&= y f(x)\otimes 1 + 1 \otimes y f(x) + \sum_{l=1}^{p-1} \sum_{m=1}^{l}  \frac{(-1)^{l-1}}{p-l} \tbinom{l}{m}  \Big[ yx^{l-m}\otimes  x^{m} + x^{m}\otimes y x^{l-m} \Big].
\end{align*}
Changing the order of the double summation $\sum_{l=1}^{p-1} \sum_{m=1}^{l} $ and replacing $l-m$ by $j$ and $m$ by $i$, we have
\[\Delta(y f(x))- y f(x)\otimes 1 - 1 \otimes y f(x)= \sum_{i=1}^{p-1} \sum_{j=0}^{p-1-i}  \frac{(-1)^{i+j-1}}{p-i-j} \tbinom{i+j}{i} \Big[ yx^{j}\otimes  x^{i} + x^{i}\otimes y x^{j} \Big],\]
which is the same as $[y\otimes 1+1\otimes y,\omega(x)]$.

iii) Let $e\in H$ be a primitive element. Since $K=u(P(H))$, we can write $e=ux+vy$ for some $u,\ v\in \field$. Thus $[x,\ e]=vy$ and $[y,\ e]=-uy$. Hence $e=0$ if $[x,\ e]=[y,\ e]=0$.
\end{proof}

\begin{lem}\label{zDelta}
There exists some $z\in H\setminus K$ such that $H$ is generated by $K$ and $z$, and the comultipilication of $z$, up to isomorphism, is one of the following
\begin{itemize}
\item[(1)] $\Delta(z)=z\otimes 1+1\otimes z+\omega(y)$,
\item[(2)] $\Delta(z)=z\otimes 1+1\otimes z+\omega(x)$,
\item[(3)] $\Delta(z)=z\otimes 1+1\otimes z- 2 x\otimes y$ for $p>2$.
\end{itemize}
Moreover, the comultiplication of $z$ remains the same if it is shifted by any primitive element.
\end{lem}
\begin{proof}
Suppose that $H$ is cocommutative.
By Theorem \ref{HCT}, there exists some $u\in H\setminus K$ such that
\begin{align*}
\Delta(u)=u\otimes 1+1\otimes u+ \omega(\alpha x +\beta y),
\end{align*}
where $\alpha,\beta \in \field$, not all zero.
If $\alpha =0$, then $\beta\ne 0$.
Letting $z=\beta^{-p} u$, we have $\Delta(z)=z\otimes 1+1\otimes z+\omega(y)$.
This gives (1).
If $\alpha \ne 0$, let  $z=\alpha^{-p} u$.
Then $\Delta(z)=z\otimes 1+1\otimes z+\omega(x+ m y)$ with $m=\beta \alpha^{-1}$.
Note that $(\ad\ (\lambda my +x))^{p-1} (my)=my$ for any $\lambda\in \k$.
Then, we have $(x+ m y)^p=x^p + (my)^p + my = x+my$. It is clear that $[x+my, y]=y$. Thus, we can assume that $\Delta(z)=z\otimes 1+1\otimes z+\omega(x)$ by a shifting $x-my$ of $x$.
This gives (2).

Now suppose that $H$ is noncocommutative.
By Proposition \ref{Liealgebrainclusion}, $\{\omega(x),\ \omega(y),\ x\otimes y\}$ is a basis of $\HL^2(\k, K)$. Then by Theorem \ref{HCT}, there exists some $u\in H\setminus K$ such that
\begin{align*}
\Delta(u)=u\otimes 1+1\otimes u+ \alpha \omega(x) + \beta \omega(y)+\gamma x\otimes y,
\end{align*}
where $\alpha,\beta, \gamma\in \field$ and $\gamma\ne 0$.
It is clear that $[\Delta(x), \alpha\omega(x)]=0$.
By Lemma \ref{Kp2calc} i), $[\Delta(x), \beta\omega(y)]=0$. Thus,
\begin{align*}
\Delta([x,u])&=[\Delta(x),\Delta(u)]\\
&=[x,u]\otimes 1+1\otimes [x,u]+ [\Delta(x),\, \gamma\,  x\otimes y] \\
&= [x,u]\otimes 1+1\otimes [x,u]+ \gamma x\otimes y.
\end{align*}
We will show that $p\neq 2$ by contradiction. Suppose $p=2$ and consider $\dim ( H_2/K_2)$. First, we have $\dim \big( H_2/K_2\big)\leq 1$ by Proposition \ref{Contraddim} and the fact that $K$ has $p$-index one in $H$.
Direct computation shows that
\[
\Delta([y,u])=[y,u]\otimes 1+1\otimes[y,u]+ \alpha(x\otimes y+ y\otimes x)+ \gamma y\otimes y.
\]
Apply $d^1$ to the elements $[x, u]$ and $[y,u]+\alpha xy$ of $H_2$, that is,
\[
d^1([x,u]) = \gamma x\otimes y \quad \text{and} \quad d^1([y,u]+\alpha xy)=\gamma y\otimes y.
\]
By Proposition \ref{Liealgebrainclusion}, $x\otimes y$ and $y\otimes y$ are linearly independent in $\HL^2(\k, K)$. Moreover, by Theorem \ref{Cohomologylemma}, $d^1$ induces an injection from $H_2/K_2$ to $\HL^2(\k, K)$.
Hence, $\dim \big( H_2/K_2\big)\ge 2$, a contradiction.
Therefore, $p>2$. Let $z =-2[x,u]/ \gamma$. Then $\Delta(z)=z\otimes 1+1\otimes z- 2 x\otimes y$, which gives (3).

In all the cases, $z\notin K$ by direct computation using PBW Theorem.
Note that the $p$-index of $K$ in $H$ is one. Hence $H$ is generated by $K$ and $z$.
The last statement of the lemma is obvious.
\end{proof}


\noindent\textbf{Proof of Theorem \ref{Kp2NCclassesAll}.} Now we assume that $H$ is generated by $K$ and some $z \in H\setminus K$ with the three types of comultiplication of $z$ described in the preceding lemma.

\textbf{Case 1.} Suppose $\Delta(z)=z\otimes 1+1\otimes z+\omega(y)$. It is clear that $[\Delta(y), \omega(y)]=0$.
By Lemma \ref{Kp2calc} i), $[\Delta(x), \omega(y)]=0$.
Hence both $[x,z]$ and $[y,z]$ are primitive and we can assume that
\[
[x,\ z]=ax+by,\quad [y,\ z]=c x+dy\quad \text{for some } \; a, b, c, d\in \field .
\]
Since $\Chara \field=p$ and $[x, y]=y$, we have $[x^p,z]=(\ad\ x)^p (z)= by$. Then $[x,z]=[x^p,z]$ implies that $a=0$. Replacing $z$ with $z+dx-by$, the commutator relations can be reduced to
\[
[x,\, y]=y,\quad [x,\ z]=0,\quad [y,\ z]=c x \quad \text{for some } \; c\in \field .
\]
Applying these relations, it follows that
\[
(xy)z=zyx+zy+cx^2+cx\quad \text{and} \quad x(yz)=zyx+zy+cx^2.
\]
Then, by associativity, we have $c=0$ or $[y,\ z]=0$. Since $\Chara \field =p$ and $[y,\ z]=0$, we have
\[
\Delta(z^p)=\Delta(z)^p=z^p\otimes 1+1\otimes z^p+\omega(y^p)=z^p\otimes 1+1\otimes z^p.
\]
Hence $z^p$ is primitive.  But $[x,\ z]=[y,\ z]=0$ implies
$[x,\ z^p]=[y,\ z^p]=0$, and so $z^p=0$ by Lemma \ref{Kp2calc} iii). This gives (B1).


\textbf{Case 2.} Suppose $\Delta(z)=z\otimes 1+1\otimes z+\omega(x)$. It is easy to see that $[x,z]$ is primitive. By Lemma \ref{Kp2calc} ii), we have
\begin{eqnarray*}
\Delta\Big([y,z]-y f(x)\Big)
&=&\left[\Delta(y),\Delta\left(z\right)\right]- \Delta(y f(x))\\
&=&[y,z]\otimes 1+1\otimes [y,z]+[y\otimes 1+1\otimes y,\omega(x)]- \Delta(y f(x))\\
&=&([y,z]- y f(x))\otimes 1 + 1 \otimes ([y,z]-y f(x)).
\end{eqnarray*}
Hence $[y,z]-y f(x)$ is also primitive. Then we can assume that
\[
[x,\ z]=ax+by,\quad [y,\ z]=y f(x)+ c x+dy\quad \text{for some } \; a, b, c, d\in \field .
\]
Similar arguments involving $[x,z]=[x^p,z]$ and the replacement of $z$ with $z+dx-by$ show that the above commutator relations can be reduced to
\[
[x,\, y]=y,\quad [x,\ z]=0,\quad [y,\ z]=y f(x) + c x \quad \text{for some } \; c\in \field .
\]
Applying these relations, we have
\begin{align*}
(xy)z&=zyx+zy+yf(x)(x+1)+cx^2+cx,\\
x(yz)&=zyx+zy+yf(x)(x+1)+cx.
\end{align*}
Hence by associativity, $c=0$ and $[y,\ z]=y f(x)$.
It remains to determine $z^p$. Again since $\Chara \field=p$, we have
\[
\Delta(z^p)=\Delta(z)^p=z^p\otimes 1+1\otimes z^p+\omega(x^p)= z^p\otimes 1+1\otimes z^p+\omega(x).
\]
Thus $z^p-z$ is primitive. Note that
\begin{align*}
[x,\ z^p-z]&=(x)(\ad\ z)^p-[x, z]=0,\\
[y,\ z^p-z]&=(y) (\ad\ z)^p- y f(x)=yf(x)^p-yf(x)=0.
\end{align*}
Hence $z^p=z$ by Lemma \ref{Kp2calc} iii). This gives (B2).


\textbf{Case 3.} Suppose $p>2$ and $\Delta(z)=z\otimes 1+1\otimes z- 2x\otimes y$. Then we have
\[\Delta([x,z])=[x,z]\otimes 1+1\otimes[x,z]-2 x\otimes y.\]
Hence $[x,z]-z$ can be written as $ax+by$ for some $a, b\in \field$.
That is, $[x,z]=z+ax+by$.
It follows that $[x^p,z]=(\ad\ x)^p(z)=z+ax+pby=z+ax$.
On the other hand, $[x^p,z]=[x,z]= z+ax+by$.
Thus, $b=0$ and $[x,z]=z+ax$.
By a shifting $z+ax$ of $z$, we can assume that $[x,z]=z$.
To determine $[y,\ z]$, we consider $[y,z]-y^2$. It follows from $[x,\ y]=y$ and $\Delta(z)=z\otimes 1+1\otimes z- 2x\otimes y$ that
\[
\Delta([y,z]-y^2) =
\big([y,z]- y^2\big)\otimes 1+1\otimes\big([y,z]-y^2\big).
\]
Then we can write $[y,z]=y^2+cx+dy$ for some $c, d\in \field$. Applying the three commutators $[x,y]=y$, $[x,z]=z$ and $[y,z]=y^2+cx+dy$, we have
\begin{align*}
x(yz)&=zyx+zy+y^2x+2y^2+cx^2+dyx+dy,\\
(xy)z&=zyx+zy+y^2x+2y^2+cx^2+dyx+2cx+2dy.
\end{align*}
By associativity, $c=d=0$, that is, $[y,z]=y^2$.

Next, we determine $z^p$. Note that $y^p=0$. By Proposition \ref{palgebra} we have
\[\Delta(z^p)=\Big(z\otimes 1+1\otimes z- 2x\otimes y\Big)^p = z^p\otimes 1+1\otimes z^p+\sum_{i=1}^{p-1}s_i,\]
where $is_i$ is the coefficient of $\lambda^{i-1}$ in $(-2x\otimes y) (\ad\ (-2\lambda x\otimes y+ z\otimes 1+1\otimes z))^{p-1}$.
Set $A_0=-2x\otimes y$.
For $n=1, 2, \ldots$, denote
\[A_n=[A_{n-1},\, -2\lambda x\otimes y+z\otimes 1+1\otimes z].\]
Then $i s_i$ is the coefficient of $\lambda^{i-1}$ in $A_{p-1}$.
Note that
\[A_1=[-2x\otimes y,\, -2\lambda x\otimes y+z\otimes 1+1\otimes z]
=-2z\otimes y-2x\otimes y^2.\]
We make the inductive assumption $A_{n}=a_{n}z\otimes y^{n}+b_{n}x\otimes y^{n+1}$ where $a_{n},b_{n}\in \field [\lambda]$, the polynomial ring in $\lambda$ over $\field$. Then
\begin{align*}
A_{n+1}=&\left[a_{n}z\otimes y^{n}+b_{n}x\otimes y^{n+1},\, -2\lambda x\otimes y+z\otimes 1+1\otimes z\right]\\
=&-2\lambda a_n[z,x]\otimes y^{n+1}+ a_nz\otimes [y^n, z]+ b_n[x, z]\otimes y^{n+1} + b_n x\otimes [y^{n+1}, z].
\end{align*}
Note that $[y^n, z]=ny^{n+1}$ by $[y,z]=y^2$. Hence
\[A_{n+1}=(2\lambda+na_{n}+b_{n}) z\otimes y^{n+1} + (n+1) b_{n} x\otimes y^{n+2}.
\]
That is, $a_{n+1}=2\lambda+n a_{n}+b_{n}$ and $b_{n+1}=(n+1) b_{n}$. Then, combining with the initial condition $a_1=b_1=-2$, we have
\[
a_{p-1}=\sum_{i=1}^{p-1} c_{i-1} \lambda^{i-1} \quad \text{for some } c_{i-1}\in \field,
\]
and
\[
A_{p-1}=a_{p-1}z\otimes y^{p-1}=\sum_{i=1}^{p-1}  c_{i-1}\, z\otimes y^{p-1} \lambda^{i-1}.\]
Therefore,
\[
\Delta(z^p)= z^p\otimes 1+1\otimes z^p+\sum_{i=1}^{p-1}s_i=z^p\otimes 1+1\otimes z^p+\sum_{i=1}^{p-1} \frac{ c_{i-1}}{i}\,  z\otimes y^{p-1}.
\]
Next, we denote $a=\sum_{i=1}^{p-1} c_{i-1}/i$ and show $a=0$ by coassociativity. Consider
{\small
\begin{align*}
\left(\Delta\otimes \Id\right)\Delta\left(z^p\right)
=&\left(\Delta\otimes \Id\right)(z^p\otimes 1+1\otimes z^p+ a z\otimes y^{p-1})\\
=&\Delta(z^p)\otimes 1+1\otimes 1\otimes z^p+ a\Delta(z)\otimes y^{p-1}\\
=&\Delta(z^p)\otimes 1+1\otimes 1\otimes z^p + a(z\otimes 1\otimes y^{p-1}+1\otimes z\otimes y^{p-1}-2x\otimes y\otimes y^{p-1}),\\
\left(\Id\otimes \Delta\right)\Delta\left(z^p\right)
=&\left(\Id\otimes \Delta \right)(z^p\otimes 1+1\otimes z^p+ a z\otimes y^{p-1})\\
=&z^p\otimes 1\otimes 1+1\otimes \Delta(z^p)+ a z\otimes \Delta(y^{p-1}).
\end{align*}
}
Note that neither $1\otimes \Delta(z^p)$ nor $z\otimes \Delta(y^{p-1})$ contains a term in $x\otimes y\otimes y^{p-1}$. Then $a=0$ by coassociativity.
Hence $\Delta(z^p)= z^p\otimes 1+1\otimes z^p$ and $z^p$ is primitive.
Moreover, $[x,z^p]=x(\ad\ z)^p=0$ and $[y,z^p]=y(\ad\ z)^p=p! y^{p+1}=0$. Thus $z^p=0$ by Lemma \ref{Kp2calc} iii).
This gives (B3).


\section{When $K$ is $p^3$-dimensional}
In this section, assume $\dim K=p^3$, that is, $H$ is primitively generated. It is sufficient to classify all restricted Lie algebras of dimension three.

Let $\g$ be a Lie algebra of dimension three (not necessarily restricted), spanned by $x,\, y,\, z$. The classification of such $\g$ is well known (see, for example, \cite[\S 3]{Strade}):
\begin{lem} \label{list3Lie}
The Lie structure of a three-dimensional Lie algebra over $\k$, spanned by $x,\, y,\, z$, is one of the following, up to isomorphism.
\begin{itemize}
\item[1)] $[x,y]=[x,z]=[y,z]=0$,
\item[2)] $[x,y]=z,[x,z]=[y,z]=0$,
\item[3)] $[x,y]=z, [x,z]=x, [y,z]=-y$,
\item[4)] $[x,y]=y,[x,z]=[y,z]=0$,
\item[5)] $[x,y]=0,[x,z]=\lambda x,[y,z]=\lambda^{-1} y$, for some $\lambda \in \k^{\times}$.
\end{itemize}
Moreover, type 1) is abelian, type 2) is Heisenberg, type 4) and type 5) are both non-semisimple and non-nilpotent, and type 3) is the only simple one.
\end{lem}
For the purpose of this section, we are only interested in the restricted Lie algebras of such types. It then remains to determine the $p$-map structure of $\g$ (if it exists). The next lemma follows directly from Corollary \ref{pgeneratedCHopf}.
\begin{lem} \label{nilpLie}
When $\g$ is abelian,  the $p$-map structure of $\g$ is one of the following, up to isomorphism.
\begin{itemize}
\item[(1)] $x^p=y,\ y^p=z,\ z^p=0$,
\item[(2)] $x^p=0,\ y^p=z,\ z^p=0$,
\item[(3)] $x^p=0,\ y^p=0,\ z^p=0$,
\item[(4)] $x^p=0,\ y^p=0,\ z^p=z$,
\item[(5)] $x^p=y,\ y^p=0,\ z^p=z$,
\item[(6)] $x^p=0,\ y^p=y,\ z^p=z$,
\item[(7)] $x^p=x,\ y^p=y,\ z^p=z$.
\end{itemize}
\end{lem}

Next, we classify nonabelian restricted Lie algebras of dimension three.
In Lemma \ref{Heisenberg} and Lemma \ref{notnilsipLie(a)}, to obtain the most simplified $p$-map structure of $\g$, we will rescale the generators $x,\, y,\, z$ when it is necessary.
That is, we will take $x=a\tilde{x}$, $y=b\tilde{y}$, and $z=c\tilde{z}$ for some $a, b, c\in \k^{\times}$ so that $\tilde{x}$, $\tilde{y}$ and $\tilde{z}$ are satisfying the same Lie bracket defining relations of $x, y, z$.
Moreover, some shifting (e.g., $x'=x+uy$, $x'=x+vz$ for some $u, v\in \k$) might be necessary too.
Once the $p$-map of $\g$ is simplified, the restricted Lie algebra $\g$ will still be represented using $x$, $y$, and $z$, after rescaling or shifting.

\begin{lem}\label{Heisenberg}
If $\g$ is Heisenberg, then the $p$-map of $\g$ is one of the following:
\begin{itemize}
\item[(1)] $x^p=0,\ y^p=0,\ z^p=0$,
\item[(2)] $x^p=0,\ y^p=0,\ z^p=z$,
\item[(3)] $x^p=z,\ y^p=0,\ z^p=0$.
\end{itemize}
If $p=2$, then the $p$-maps in \emph{(1)} and \emph{(3)} are isomorphic. If $p \neq 2$, then there are three different isomorphism classes.
\end{lem}
\begin{proof}
As in Lemma \ref{list3Lie} 2), the Lie structure of $\g$ is given by $[x,y]=z,[x,z]=[y,z]=0$.
It is straight forward to verify that $(\ad\ x)^p$, $(\ad\ y)^p$, $(\ad\ z)^p$ vanish on $\g$. Then $x^p,\ y^p,\ z^p$ are all in the center of $\g$.
Hence, $x^p=\alpha z$, $y^p=\beta z$ and $z^p=\gamma z$ for some $\alpha, \beta, \gamma\in \field$. The case (1) is obvious if $\alpha=\beta=\gamma=0$.

When $\alpha=\beta=0$ and $\gamma\ne 0$, we can assume $\gamma=1$ by rescaling with the scalars $ab=c$ and $c^{p-1}=\gamma$. This gives (2).

When $\alpha\ne 0$ and $\beta=\gamma=0$, we can assume that $\alpha=1$ by rescaling with $ab=c$ and $a^p=\alpha c$. This leads to (3).
Due to the symmetry of $x$ and $y$, the case when $\beta\ne 0$ and $\alpha=\gamma=0$ gives a restricted Lie algebra isomorphic to the one  with the $p$-map in (3).

When $\alpha\ne 0$, $\beta=0$ and $\gamma\ne 0$ (or symmetrically $\alpha= 0$, $\beta\ne 0$ and $\gamma\ne 0$), we can assume that $\alpha=1$, $\beta=0$ and $\gamma=1$ by rescaling with $c^{p-1}=\gamma$, $a^p=\alpha c$, and $ab=c$. That is, $x^p=z,\ y^p=0,\ z^p=z$. Now set $x'=x-z$. Then $(x')^p=(x-z)^p=0$, $[x', y]=z$, and $[x',z]=0$. Therefore, these cases also give (2).

The remaining case is when both $\alpha$ and $\beta$ are nonzero.
Set $x'=x+u y$ for some $u\in \field$. Then
\[
(x')^p=(x+u y)^p=x^p+u^py^p+\sum_{i=1}^{p-1}s_i,
\]
where $is_i$ is the coefficient of $\lambda^{i-1}$ in $x (\ad\ (\lambda x+uy))^{p-1}$.
If $p>2$, then $(x')^p=x^p+(u y)^p=(\alpha +u^p\beta) z$, since $[x,z]=[y,z]=0$.
We can choose $u$ such that $(x')^p=0$.
If $p=2$, then $(x')^2=x^2+u^2y^2+[x, x+uy]=(\alpha+u+u^2\beta)z$.
Again, we can choose $u$ properly such that $(x')^p=0$.
Note that $[x', y]=z$ and $[x',z]=0$.
This case reduces to one of the previous cases and does not give any new isomorphism classes with different $p$-maps.

Lastly, it is clear that the $p$-map in (2) produces a non-$p$-nilpotent restricted Lie algebra and so gives a distinct isomorphism class. When $p>2$, the $p$-map in (1) is trivial, i.e., any linear combination of $x, y, z$ maps to zero under the $p$-map. Hence the $p$-maps in (1) and (3) give distinct isomorphism classes of $\g$ for $p>2$. When $p=2$, the $p$-map in (1) is not trivial, since $(x+y)^2=z$. Then in this case the $p$-maps in (1) and (3) are isomorphic via the mapping $x\mapsto x+y$, $y\mapsto y$, $z\mapsto z$. This completes the proof.
\end{proof}

\begin{lem}
Assume $\g$ is simple. If $p=2$, then the $p$-map does not exist. If $p>2$, then the $p$-map is given by $x^p=y^p=0,\, z^p=z$, up to isomorphism.
\end{lem}
\begin{proof}
By Lemma \ref{list3Lie} 3), $[x,y]=z, [x,z]=x, [y,z]=-y$. For any $\alpha, \beta, \gamma \in \field$, it holds that $[\alpha x+ \beta y+ \gamma z,\, y]=\alpha z+\gamma y$. But $[x^2,\, y]=(\ad\ x)^2(y) =[x,\, [x,\, y]]=[x,\, z]=x$. Hence there is no restricted map when $p=2$. Now assume that $p >2$. Since $(\ad\ x)^p$ and $(\ad\ y)^p$ vanish on $\g$, we have $x^p=y^p=0$. Moreover, it follows from $[x,\, z^p]=x$ and $[y,\,z^p]=-y$ that $z^p=z$. This completes the proof.
\end{proof}

\begin{lem}\label{notnilsipLie(a)}
If $\g$ has the Lie structure $[x,y]=y,[x,z]=[y,z]=0$ as in Lemma \ref{list3Lie}, 4), then the $p$-map of $\g$ is one of the following, up to isomorphism.
\begin{itemize}
\item[(1)] $x^p=x,\ y^p=0,\ z^p=0$,
\item[(2)] $x^p=x,\ y^p=z,\ z^p=z$,
\item[(3)] $x^p=x,\ y^p=z,\ z^p=0$,
\item[(4)] $x^p=x,\ y^p=0,\ z^p=z$.
\end{itemize}
\end{lem}
\begin{proof}
Let $w=k_1 x + k_2 y + k_3 z$ be an element of $\g$ for $k_1, k_2, k_3\in \k$.
Then $[w,\,x]=-k_2 y$, $[w,\, y]=k_1y$, and $[w,\, z]=0$.
In particular, the center of $\g$ is $\k z$.
Note that $[x^p,\, x]=[x^p,\, z]=0$, $[x^p,\, y]=y$, and $y^p, z^p$ are in the center of $\g$.
Then, $x^p=x+\alpha z$, $y^p=\beta z$, and $z^p=\gamma z$ for some $\alpha, \beta, \gamma \in \field$.
Consider the shifting $x+v z$, where $v$ is a scalar such that $\alpha + v^p \gamma= v$.
Then, $(x+v z)^p=x+v z$, $[x+v z,\, y]=y$, and $[x+v z,\, z]=0$.
Thus, we can assume that $\alpha=0$, i.e., $x^p=x$.

Obviously, $\beta=\gamma=0$ gives (1).
If both $\beta$ and $\gamma$ are nonzero, we can assume $\beta=\gamma=1$ by rescaling with $c^p=\gamma c$ and $b^p=\beta c$. This gives (2).
For the remaining two cases, we can assume $\beta=1,\, \gamma=0$ or  $\beta=0,\, \gamma=1$ by rescaling with $b^p=\beta c$ or $c^p=\gamma c$, and so obtain (3) and (4), respectively.

It can be shown as follows that the four $p$-maps give four distinct isomorphism classes of $\g$.
Note that $[\g,\g]=\k y$.
Then $[\g,\g]^p=0$ for (1) and (4) while $[\g,\g]^p=\k z$ for (2) and (3).
Next, $\dim \g^p=1$ for (1) but $\dim \g^p=2$ for (4).
Lastly, (2) produces a restricted Lie algebra with $p$-nilpotent center but (3) does not.
\end{proof}

\begin{lem}\label{notnilsipLie(b)}
Suppose $\g$ has the Lie structure $[x,y]=0,[x,z]=\lambda x,[y,z]=\lambda^{-1} y$, for some $\lambda \in \k^{\times}$ as in Lemma \ref{list3Lie}, 5). If $\g$ is a restricted Lie algebra, then $\lambda^{p-1}=\pm 1$, and the $p$-map of $\g$ is, up to isomorphism,
\[x^p=0,\, y^p=0,\, z^p=\lambda^{p-1} z. \]
When $p>2$, there are $p+1$ isomorphism classes of such restricted Lie algebras.
When $p=2$, there is only one such restricted Lie algebra.
\end{lem}
\begin{proof}
One can check that the center of $\g$ is trivial.
It is also clear that $x^p$ and $y^p$ are both in the center of $\g$.
Hence, $x^p=y^p=0$.
It follows from the given Lie bracket relation that $[x,\, z^p]=\lambda^p x$, $[y,\,z^p]=\lambda^{-p} y$, $[z^p,\,z]=0$.
Suppose that $z^p=\alpha x+\beta y+\gamma z$ for some $\alpha, \beta, \gamma\in \field$.
Then $[x,\,z^p]=\lambda \gamma x$, $[y,\, z^p]=\lambda^{-1} \gamma y$, and $[z, z^p]=- \lambda \alpha x-\lambda^{-1}\beta y$.
Hence, $\alpha=\beta=0$, $\lambda^{p-1}=\gamma=\lambda^{-(p-1)}$, which further implies that $\gamma=\pm1$.

When $p>2$, distinguished by the roots of 1 or the roots of $-1$, there are two types of $p$-maps producing two classes of restricted Lie algebras.
Denote them by $\g_1(\lambda)$ and $\g_{-1}(\mu)$, respectively, where $\lambda^{p-1}=1$ and $\mu^{p-1}=-1$.
It is clear that $\g_{1}(\lambda) \ncong \g_{-1}(\mu)$ for any $\lambda$ and $\mu$.
Due to the symmetry of $x$ and $y$, we have $\g_{1}(\lambda)\cong \g_{1}(\lambda')$ if and only if $\lambda'=\lambda^{-1}$ or $\lambda'=\lambda$.
The same holds for $\g_{-1}(\mu)$.
Then, for any fixed $p>2$, there are $p+1$ isomorphism classes for restricted Lie algebras with the given Lie structure.
The result for $p=2$ is clear.
\end{proof}

\noindent\textbf{Proof of Theorem \ref{Kdimp3classesAll}.} The theorem follows from Lemmas \ref{nilpLie}-\ref{notnilsipLie(b)}.

\begin{remark}The isomorphism classes in Theorem \ref{Kdimp3classesAll} are grouped by semisimple, local, and neither. By \cite[Theorem 2.3.3]{MO93} and Corollary \ref{pgeneratedCHopf}, type (C1) is the only semisimple one. By Theorem \ref{NPLA}, finite-dimensional primitively generated connected Hopf algebras are local if and only if all primitives are nilpotent. It is clear that types (C2)-(C6) are the only ones with $x, y, z$ all nilpotent. The remaining types (C7)-(C16) are neither semisimple nor local.
\end{remark}

\section{Algebra structure of (B2)}
In this section, we want to compare the algebra structure of (B2) type in Theorem \ref{Kp2NCclassesAll} with restricted enveloping algebras. The motivation comes from the result by D.-G. Wang, J.J. Zhang, and G. Zhuang \cite{WZZ2} on connected affine Hopf algebras in characteristic zero described in the Introduction. Note that restricted enveloping algebras of dimension $p^3$ are classified in Theorem \ref{Kdimp3classesAll} as Hopf algebras of type C.  For convenience, we denote by $A$ the algebra described as (B2), namely $A$ is the quotient of the free algebra generated by three variables $x,y,z$ by the ideal generated by:
\begin{align}\label{RelationA}
[x, y]-y,\ [x, z],\ [y, z]-yf(x), x^p-x,\ y^p,\ \mbox{and}\ z^p-z,\tag{6a}
\end{align}
where $f(x):=\sum_{i=1}^{p-1}(-1)^{i-1}(p-i)^{-1}x^i$. For the sake of the notation, we will sometimes simply write $f$ for the element $f (x)$.

\begin{lem}\label{trivialCenter}
The center of the algebra $A$ is trivial.
\end{lem}
\begin{proof}
As stated before, $A$ is the quotient of the free algebra generated by three variables $x,y,z$ by the ideal generated by \eqref{RelationA}.
Therefore, by Diamond Lemma \cite{Beg}, $A$ has a basis $\{z^iy^jx^k\, |\, 0\le i,j,k\le p-1\}$.  Moreover, we use the lexicographical order on the index set $I=\{(i,j,k)\, |\, 0\le i,j,k\le p-1\}$ defined as $(\alpha_1,\alpha_2,\alpha_3)<_{\mbox{lex}}(\beta_1,\beta_2,\beta_3)$ if and only if
\[
(\alpha_1<\beta_1), \ (\alpha_1=\beta_1,\ \alpha_2<\beta_2),\ \mbox{or}\ (\alpha_1=\beta_1,\ \alpha_2=\beta_2,\, \alpha_3<\beta_3).
\]

Suppose, the element $C=\sum_{\alpha\in I} a_\alpha z^{\alpha_1}y^{\alpha_2}x^{\alpha_3}$ belongs to the center of $A$. We can check directly the following identities in $A$ inductively:
\begin{align}
[x,y^n]&=ny^n, \tag{6.1a}\label{6.1a}\\
[y,x^n]&=-y\big((x+1)^n-x^n\big), \tag{6.1b}\label{6.1b}\\
[y,z^n]&={n\choose1}z^{n-1}yf+{n\choose 2}z^{n-2}yf^2+\cdots+{n\choose 0}yf^n, \tag{6.1c}\label{6.1c}
\end{align}
for any integer $n\ge 1$. By using the condition $[x,C]=0$ and \eqref{6.1a}, we have
\[
\sum_{\alpha\in I} \alpha_2a_\alpha z^{\alpha_1}y^{\alpha_2}x^{\alpha_3}=0.
\]
Then $\alpha_2=0$ whenever $a_\alpha\neq 0$. So we can assume that $C=\sum_{\alpha\in J} a_\alpha z^{\alpha_1}x^{\alpha_3}$ where $J$ is the index subset of $I$ with the zero middle index. Furthermore the condition $[y,C]=0$ combining \eqref{6.1b} and \eqref{6.1c} implies that
\begin{align}\label{centerA}\tag{6.1d}
[y,C]=&\sum_{\alpha\in J}a_\alpha[y,z^{\alpha_1}]x^{\alpha_3}+\sum_{\alpha\in J}a_\alpha z^{\alpha_1}[y,x^{\alpha_3}]\\
=&\sum_{\alpha\in J}a_\alpha\left[{\alpha_1\choose1}z^{\alpha_1-1}yf+{\alpha_1\choose 2}z^{\alpha_1-2}yf^2+\cdots+{\alpha_1\choose 0}yf^{\alpha_1}\right]x^{\alpha_3}\notag\\&-\sum_{\alpha\in J}a_\alpha z^{\alpha_1}\left[{\alpha_3\choose1}x^{\alpha_3-1}+{\alpha_3\choose 2}x^{\alpha_3-2}+\cdots+{\alpha_3\choose 0}\right]\notag\\
=&0.\notag
\end{align}
Choose the leading term of $C$ under the lexicographical order and denote it by $az^mx^n$ where $a\neq 0$. From \eqref{centerA}, it is clear that the leading term of $[y,C]$ is $-naz^{m}x^{n-1}$. Hence we have $n=0$. Suppose $m\ge 1$. Then we can rewrite $C$ as
\[az^m+bz^{m-1}g(x)+\delta,\]
where $b\in \k$, $g(x)$ is an element of $A$ only in terms of $x$ and $\delta$ is the tail term satisfying $\delta<_{\mbox{lex}} z^{m-1}$. By using \eqref{centerA} again, we see the leading term of $[y,C]$ now becomes
\[
amz^{m-1}yf+bz^{m-1}[y,g(x)]=amz^{m-1}yf-bz^{m-1}y[g(x+1)-g(x)].
\]
Since $am\neq 0$, we have
\begin{align}\label{relationf}\tag{6.1e}
f(x)=(b/am)[g(x+1)-g(x)]
\end{align}
in the subalgebra $S$ of $A$ generated by $x$. It is easy to see that $S=k[x]/(x^p-x)$. We fix a basis $\{1,x,x^2,\cdots,x^{p-1}\}$ for $S$ and write
\[
g(x)=a_{p-1}x^{p-1}+a_{p-2}x^{p-2}+\cdots+a_1x+a_0,\]
for some coefficients $a_i\in \k$. Then by the definition of $f(x)$ in relations \eqref{RelationA}, we see that \eqref{relationf} is impossible since $g(x+1)-g(x)$ does not have the leading term $x^{p-1}$. We have a contradiction. Then $m$ must be zero and we can further assume that $C=g(x)$. Suppose $g(x)$ is not a constant. Since $\k$ is algebraically closed, $g(x)$ has at least one root $\alpha\in \k$. By \eqref{6.1b} again, we know $[y,g(x)]=-y[g(x+1)-g(x)]$. So $g(x+1)=g(x)$ and $g(x)$ has $p$-distinct roots $\alpha,\alpha+1,\cdots,\alpha+{p-1}$. It is a contradiction since $g(x)$ has degree less than $p$. So $C\in \k$ and the center of $A$ is trivial.
\end{proof}

In the following, we point out the Jacobson radical of $A$ and the corresponding associated graded ring, which are easy to check.

\begin{prop}\label{quotientJ}
The Jacobson radical $J_A$ of $A$ is generated by $y$ and $A/J_A\cong \underbrace{k\times k\times \cdots \times k}_{2p}$. Moreover for the associated graded ring of $A$ with respect to $J_A$-adic filtration, we have
\begin{align*}
\gr_{J_A}A\cong A,\ \mbox{for}\ \deg x=\deg z=0,\deg y=1.\\
\end{align*}
\end{prop}

\begin{lem}\label{lemPLie}
Let $\mathfrak g$ be a finite-dimensional restricted Lie algebra over $\k$. If every irreducible restricted representation of $\mathfrak g$ is one-dimensional, then $[\mathfrak g,\mathfrak g]$ is $p$-nilpotent in $\mathfrak g$.
\end{lem}
\begin{proof}
Denote by $u(\mathfrak g)$ the restricted universal enveloping algebra of $\mathfrak g$. By its construction, irreducible restricted representations of $\mathfrak g$ are in one-to-one correspondence with simple modules over $u(\mathfrak g)$. Let $M$ be any simple module over $u(\mathfrak g)$, which is one-dimensional by assumption. Then every element of $\mathfrak g$ acts on $M$ by a scalar. Hence $[\mathfrak g,\mathfrak g]M=0$, which implies that the commutator $[\mathfrak g,\mathfrak g]$ annihilates all the simple modules over $u(\mathfrak g)$. So $[\mathfrak g,\mathfrak g]$ lies inside the Jacobson radical of $u(\mathfrak g)$ by definition. Since $u(\mathfrak g)$ is finite-dimensional, the Jacobson radical of $u(\mathfrak g)$ is nilpotent. Thus we have $[\mathfrak g,\mathfrak g]$ is $p$-nilpotent after we translate everything back into $\mathfrak g$.
\end{proof}

We are now ready to show that there exists a finite-dimensional connected Hopf algebra in positive characteristic, which as an algebra is not isomorphic to any restricted universal enveloping algebras.
\begin{prop}
Let $\k$ be an algebraically closed field of characteristic $p>0$.
As an algebra over $\k$, the (B2) type is not isomorphic to any restricted universal enveloping algebras.
\end{prop}

\begin{proof}
From the classification of restricted universal enveloping algebras in Theorem \ref{Kdimp3classesAll}, it is clear to see that (C1)-(C14) all have non-trivial center. Then by Lemma \ref{trivialCenter}, it suffices to compare (B2) with (C15) and (C16). Every irreducible representation of (B2) is one-dimensional by Proposition \ref{quotientJ}. But in (C15), we see that $z\in [\mathfrak g,\mathfrak g]$, which is not $p$-nilpotent. Hence (B2) and (C15) cannot be isomorphic as algebras by Lemma \ref{lemPLie}. For (C16), its Jacobson radical is generated by $x,y$, whose quotient algebra is isomorphic to $\underbrace{k\times k\times \cdots \times k}_{p}$. Again by Proposition \ref{quotientJ}, we know (B2) and (C16) cannot be isomorphic. This completes the proof.
\end{proof}

\begin{Ack}
We would like to acknowledge Professor James Zhang for his help during the preparation of this manuscript. This research work was done partially during the second author's visit to the Department of Mathematics at the University of Washington in summer 2013. She is grateful for Professor James Zhang's invitation and the support from his NSF grant [DMS0855743]. The authors would like to express their gratitude to the referee for a careful reading of the first version of the paper and for his/her valuable suggestion.
\end{Ack}

\appendix
\section{Abelian restricted  Lie algebras}
The following result about finite-dimensional abelian restricted  Lie algebras can be easily derived from the discussion in \cite[Chapter V, \S 8]{Jac}.

\begin{prop}\cite[Chapter V, \S 8]{Jac}\label{basisabelianp}
Let $\mathfrak g$ be an abelian restricted Lie algebra of dimension $m< \infty$. Then there exists a set of elements $\{x_1,\ldots,x_d, y_1,\ldots, y_s\}\subset\g$, such that
\[
\g=\Big( \oplus_{i=1}^d  (\k x_i \oplus \k x_i^p \oplus \ldots \oplus \k x_i^{p^{n_i-1}}) \Big)\oplus \Big( \oplus_{j=1}^s \k y_j \Big)
\]
where $0\leq d,\ s \leq m$ are integers, $x_i^{p^{n_i}}=0$, $y_j^p=y_j$, for all $i$ and $j$, and $\sum_{i=1}^d n_i +s=m$.
\end{prop}

\begin{proof}
Following \cite[P.192]{Jac}, we consider the noncommutative polynomial ring \[\Phi=\{\alpha_0+\alpha_1t+\cdots +\alpha_nt^n\; |\; \alpha_i\in\field\}.\]
The indeterminate $t$ satisfies $t\alpha=\alpha^pt$ for any $\alpha \in \field$,
and so $t$ can be viewed as the restricted $p$-map on $\g$.
Note that $\Phi$ is a principal ideal domain.
By \cite[Ex. 19]{Jac}, $\g$ has a $p$-invariant decomposition $\g=\n\oplus \s$,
where ${\s}^{p}=\s$ and ${\n}^{p^n}=0$ for some integer $n$ sufficiently large.

When $\n\neq 0$, as a module over $\Phi$, it is a direct sum of torsion cyclic modules.
That is, $\n \cong \oplus_i \Phi/(t^{n_i})$ with $\sum n_i= \dim \n$.
For each $1\leq i \leq d$, let $x_i$ be the image of the generator of $\Phi/(t^{n_i})$ in $\n$.
Then
\[
\bigcup_{i=1}^d \Big\{ x_i,x_i^{p},\ldots,x_i^{p^{n_i-1}}\; |\; x_i^{p^{n_i}}=0 \Big\}
\]
is a basis for $\n$.
When $\s \ne 0$, let $s=\dim \s$. By \cite[Chapter V, \S 8, Theorem 13]{Jac}, a basis for $\s$ is $y_1,y_2,\ldots, y_s$ with $y_i^p=y_i$.
The proof is completed by combining the bases of $\s$ and $\n$ together.
\end{proof}

\begin{rem}\label{semisimple}
By \cite[Theorem 2.3.3]{MO93}, $u(\s)$ is semisimple. Then, by \cite[Corollary 2.3.5]{MO93} and the fact that $\k$ is algebraically closed, we have
\[ u(\s)\cong (\k (\Z_p)^s)^*\cong \k [y]\big/(y^{p^s}-y),\]
where $s=\dim \s$. Denote the image of $y$ in $\s$ by $x$.
Then $\{x,\, x^p,\, \ldots,\, x^{p^{s-1}}\}$ with $x^{p^s}=x$ is another basis of $\s$.
\end{rem}

Now the classification of finite-dimensional primitively generated commutative Hopf algebras follows immediately.

\begin{cor}\label{pgeneratedCHopf}
Any $p^m$-dimensional primitively generated commutative Hopf algebra over $\k$ is isomorphic to a Hopf algebra in the form of
\[
\k[x_1,\ldots,x_d, y_1,\ldots, y_s]/(x_1^{p^{n_1}},\ldots, x_d^{p^{n_d}},\, y_1^p-y_1,\ldots, y_s^p-y_s),
\]
where all $x_i$'s and $y_j$'s are primitive elements, $n_i$'s, $d$, and $s$ are integers such that $\sum_i^d n_i+s=m$. In particular, such a Hopf algebra is semisimple if and only if $s=m$.
\end{cor}

\begin{rem}
From Proposition \ref{basisabelianp} and Corollary \ref{pgeneratedCHopf}, we see that the number $N(m)$ of isomorphism classes of $m$-dimensional abelian restricted Lie algebra (\emph{resp.}, $p^m$-dimensional primitively generated commutative Hopf algebras) is just the partial sums of partition functions of positive integers, that is, $N(m)=\sum_{n=0}^m\, P(n)$, where $P(n)$ is the number of ways to write $n$ as a sum of non-negative integers, regardless the order of summands.
\end{rem}

\end{document}